\documentclass[11pt,reqno,oneside]{amsart}

\usepackage{graphicx}
\usepackage{amssymb}
\usepackage{epstopdf}
\usepackage{cases}
\usepackage[abbrev,nobysame]{amsrefs}
\usepackage{hyperref}
\usepackage[abbrev,nobysame]{amsrefs} \AtBeginDocument{\def\MR#1{}} 

\usepackage[a4paper, total={6in, 9in}]{geometry}

\usepackage{amsmath,amsfonts,amsthm,mathrsfs,amssymb,cite}
\usepackage[usenames]{color}

\newtheorem{thm}{Theorem}[section]

\newtheorem{lem}{Lemma}[section]
\newtheorem{prop}{Proposition}[section]
\theoremstyle{definition}
\newtheorem{defn}{Definition}[section]

\theoremstyle{remark}

\newtheorem{rem}{Remark}[section]
\numberwithin{equation}{section}

\numberwithin{equation}{section}

\usepackage{tikz}
\allowdisplaybreaks

\newcommand{\R}{\mathbb{R}} \newcommand{\mathR}{\mathbb{R}}
\newcommand{\Rn}{{\mathR^n}}

\newcommand{\scrD}{\mathscr{D}}
\newcommand{\calF}{\mathcal{F}}

\newcommand{\scrS}{\mathscr{S}}
\newcommand{\calC}{\mathcal{C}}

\newcommand{\supp}{\mathop{\rm supp}}
\newcommand{\dist}{\mathop{\rm dist}}
\newcommand{\diam}{\mathop{\rm diam}}

\newcommand{\ass}{\textrm{~a.s.~}}

\newcommand{\dif}[1]{\,\mathrm{d}{#1}} 
\newcommand{\nrm}[2][]{ \| {#2} \|_{#1}}
\newcommand{\agl}[1][\cdot]{ \langle {#1} \rangle}

\usepackage{tikz}
\newcommand{\Rk}{{\mathcal{R}_{k}}}

\newcounter{saveeqn}

\title[Recent progress of single-realization]{On recent progress of single-realization recoveries of random Schr\"odinger systems}

\author{Shiqi Ma}
\address{Department of Mathematics and Statistics, University of Jyv\"askyl\"a, Finland}
\email{mashiqi01@gmail.com, shiqi.s.ma@jyu.fi}

\begin{document}

\begin{abstract}
	
We consider the recovery of some statistical quantities by using the near-field or far-field data in quantum scattering generated under a single realization of the randomness. We survey the recent main progress in the literature and point out the similarity among the existing results. The methodologies in the reformulation of the forward problems are also investigated. We consider two separate cases of using the near-field and far-field data, and discuss the key ideas of obtaining some crucial asymptotic estimates. We pay special attention on the use of the theory of pseudodifferential operators and microlocal analysis needed in the proofs.

	\medskip

	
	
\end{abstract}

\maketitle

\section{Introduction} \label{sec:intro-RSRe2020}

\subsection{Mathematical formulations} \label{subsec:formu-RSRe2020}

In this paper, we mainly focus on the random inverse problems associated with the following time-harmonic Schr\"odinger system
\begin{equation} \label{eq:1a-RSRe2020}
	\displaystyle{ (-\Delta - E + \text{potential})\, u(x) = \text{source}, \quad x \in \Rn, }
\end{equation}
where $E$ is the energy level, $n$ is the dimension, and the ``source'' and the ``potential'' in 
\eqref{eq:1a-RSRe2020}
shall be specified later in this paper.
In some cases we may impose incident waves to the system to obtain more useful information, thus
\begin{equation} \label{eq:1b-RSRe2020}
	\displaystyle{ u(x) = \alpha \cdot u^{in}(x) + u^{sc}(x) }
\end{equation}
where $\alpha$ takes the value of either $0$ or $1$ corresponding to impose or suppress the incident wave, respectively.
The corresponding data are thus called passive or active measurements, respectively.
Moreover, we shall impose the Sommerfeld radiation condition \cite{colton2012inverse}
\begin{equation} \label{eq:1c-RSRe2020}
	\displaystyle{ \lim_{r\rightarrow\infty} r\left(\frac{\partial u^{sc}}{\partial r}-i\sqrt{E} u^{sc} \right)=0,\quad r:=|x|, }
\end{equation}
that characterizes the outgoing nature of the scattered field $u^{sc}$. 
The system \eqref{eq:1a-RSRe2020}-\eqref{eq:1c-RSRe2020} describes the quantum scattering \cite{eskin2011lectures,griffiths2016introduction} associated with a source and a potential at the energy level $E$.
Later we following the convention to use $k := \sqrt{E}$ to signify the frequency at which the system in acting on.

Under different assumptions of the potential and source, of the dimension, and of the incident wave, the regularity of the Schr\"odinger system \eqref{eq:1a-RSRe2020}-\eqref{eq:1c-RSRe2020} behaves differently and calls for different techniques for recovery procedure.
The randomness of the Schr\"odinger system \eqref{eq:1a-RSRe2020}-\eqref{eq:1c-RSRe2020} can present either in the potential, in the source, or in both the potential and source.
In this paper we shall investigate all of these three cases, consider the results in the literature and give details of part of these results.

There are rich literature on the inverse scattering problem using either passive and active measurements as data.
Fox a fixed potential, the recovery of the deterministic unknown source of the system is called the inverse source problem.
For the theoretical analysis and computational methods of the inverse source problems, readers may refer to \cite{bao2010multi, Bsource, BL2018, ClaKli, WangGuo17, Zhang2015} and references therein as references.
The simultaneous recovery of deterministic unknown source and potential are also studied in the literature.
In \cite{KM1, liu2015determining}, the authors considered simultaneous recovery of an unknown source and its surrounding medium parameter. 
This type of inverse problems also arise in the deterministic magnetic anomaly detections using geomagnetic monitoring \cite{Deng2019onident, Deng2020onident} with passive measurements. 
While \cite{Deng2019onident, Deng2020onident, KM1, liu2015determining} focus on deterministic setting with passive measurements, the works \cite{bao2016inverse, Lu1, Yuan1, LiHelinLiinverse2018, LassasA, Lassas2008, Blomgren2002, Borcea2006} pay attention to random settings.
We are particularly interested in the case with a \emph{single} realization of the random sample.
The single-realization recovery has been studied in the literature.
In this paper we mainly focus on \cite{LassasA, Lassas2008, caro2016inverse, LiHelinLiinverse2018, li2019inverse, li2020inverse, llm2018random,
ma2020determining, ma2019determining}, and we shall give these works detailed introductions as follows.

In \cite{LassasA, Lassas2008}, Lassas et.~al.~considered the inverse scattering problem for the two-dimensional random Schr\"odinger system
$(-\Delta - k^2 - q(x,\omega)) u(x,k,\omega) = \delta_y, \ x \in \R^2$ which is incited by point sources $u^{in}(x) = \frac {\mathrm{i}} {4} H_0^{(1)}(k|x-y|)$;
the $H_0^{(1)}$ is the Hankel function for the first kind, and the origin $y$ of this source are located outside the support of the potential.
The potential $q(x,\omega)$ is a micro-locally isotropic generalized Gaussian field (\emph{migr} field) with compact support.
The definition of the migr field can be found in Definition \ref{defn:migr-RSRe2020}.
They introduced the so-called rough strength $\mu(x)$, which is the informative part of the principal symbol $\mu(x) |\xi|^{-m}$ of the covariance operator.
The $-m$ in $\mu(x) |\xi|^{-m}$ is the rough order of the random potential.
The main result in their work states that under a single measurement of the random field inside a measurement domain, the rough strength can be recovered.

In 2019, Caro et.~al.~\cite{caro2016inverse} considered an inverse scattering problem for an $n$-dimensional ($n \geq 2$) random Schr\"odinger system
$(-\Delta - k^2 - q(x,\omega)) u(x,k,\omega) = 0, \ x \in \Rn$ with incident wave being the plane wave, i.e.~$u$ is incited by the point sources $u^{in}(x) = e^{ikd \cdot s}$;
$d$ is the incident direction.
Again, the potential $q$ is assumed to be a migr field with compact support.
The main result of the paper is as follows: 
they used the backscattering far-field pattern and recovered the rough strength $\mu(x)$ almost surely, under a single realization of the randomness.

In \cite{LiHelinLiinverse2018}, Li, et.~al.~studied the case where the potential is zero and the source is migr field.
In \cite{li2019inverse} Li, et.~al.~studied the same setting but with the energy level $E$ replaced by $(k^2 + i\sigma k)$ where the $\sigma$ is the attenuation.
The random source term considered is constructed as a migr field.
The system has been changed to Helmholtz system in \cite{li2019inverse} but the underlying mathematical equation is uniform with the Schr\"odinger's equation.
The authors studied the regularity of the random source and gave the well-posedness  of the direct problem.
Then they represented the solution of the system as the convolution between the fundamental solution and the random source.
By truncating the fundamental solution, they indicated that the rough strength can be recovered by utilizing the correspondingly truncated solution.
Further, the authors used calculus of symbols to recover the rough strength.

Then in \cite{li2020inverse}, Li, et.~al.~further extended their study to Maxwell's equation.
The recovery procedure in these three works share the same idea--the leading order term in the Bonn expansion gives the recovery of the desired statistics while these higher order terms converge to zero.
The proof of these converges involve the utilization of Fourier integral operator.
We shall give an detailed explanation in Section \ref{sec:Matti-RSRe2020} of this technique.

In \cite{llm2018random}, the authors consider direct and inverse scattering for \eqref{eq:1a-RSRe2020}--\eqref{eq:1c-RSRe2020}
with a deterministic potential and a random source. 
The random source is a generalized Gaussian random field with local mean value function and local variance function, which are assumed to be bounded and compactly supported.
The well-posedness of the direct scattering of the system has been formulated in weighted $L^2$ space.
Then inverse scattering is studied and recovery formula of the variance function is obtained, and the uniqueness recovery of the potential is given.
The authors used both passive and active measurement to recover to unknowns.
Passive measurements refers to the scattering data generated only by the unknown source ($\alpha$ is set to be 0 in \eqref{eq:1b-RSRe2020});
active measurements refers to the scattering data generated by both the source and the incident wave ($\alpha$ is set to be 1 in \eqref{eq:1b-RSRe2020}).
To recover the variance function, only the passive measurements is needed, while the unique recovery of the potential needs active measurements.

In \cite{ma2020determining}, the authors extended the work \cite{llm2018random} to the case where the source is a migr field.
The direct scattering problem is formulated in a similar manner as in \cite{llm2018random}, while the technique used in the inverse scattering problem differs from that of \cite{llm2018random}.
In order to analyze the asymptotics of higher order terms in the Bonn expansion corresponding to the migr fields, stationary phase lemma and pseudodifferential operator are utilized.

Then the authors extended the work \cite{ma2020determining} to the case where both the potential and the source are random (of migr type), and the extended result is presented in \cite{ma2019determining}.
The results between \cite{llm2018random} and \cite{ma2020determining,  ma2019determining} have two major differences.
First, in \cite{llm2018random} the random part of the source is assumed to be a Gaussian white noise, 
while in \cite{ma2019determining} the potential and the source are assumed to be migr fields. The migr field can fit larger range of randomness 
by tuning its rough order and rough strength.
Second, in \cite{ma2019determining} both the source and potential are random, while in \cite{ma2020determining} the potential is assumed to be deterministic.
These two facts make \cite{ma2019determining} much more challenging than that in \cite{ma2020determining}. 
The techniques used in the estimates of higher order terms in \cite{ma2019determining} are pseudodifferential operators and microlocal analysis
and we shall give a detailed treatment in Section \ref{sec:Ma-RSRe2020}.

Although the techniques used in \cite{llm2018random, ma2020determining, ma2019determining} are different, the recovery formulae fall into the same pattern.
The thesis \cite{PhDThesisShiqiMa} partially collected these three works\cite{llm2018random,  ma2020determining, ma2019determining} and readers may refer to the thesis for a more coherent discussion on this topic.

\subsection{Summarization of the main results} \label{subsec:MainRst-RSRe2020}

In this paper we mainly pay attention to two types of random model, the Gaussian white noise and the migr field.
The Gaussian white noise is well-known and readers may refer to Section 2.1 in  \cite{llm2018random} for more details.
Here we give a brief introduction to the migr field.
We assume $f$ to be a generalized Gaussian random distribution of the microlocally isotropic type (cf.~Definition \ref{defn:migr-RSRe2020}).
It means that $f(\cdot,\omega)$ is a random distribution and the mapping 
	\begin{equation*}
	\omega \in \Omega \ \mapsto \ \agl[f(\cdot,\omega),\varphi] \in \mathbb C,\quad \varphi\in\scrS(\Rn),
	\end{equation*}
	is a Gaussian random variable whose probabilistic measure depends on the test function $\varphi$. Here and also in what follows, $\scrS(\Rn)$ stands for the Schwartz space. Since both $\agl[f(\cdot,\omega),\varphi]$ and $\agl[f(\cdot,\omega),\psi]$ are random variables for $\varphi$, $\psi \in \scrS(\Rn)$,
	from a statistical point of view, the covariance between these two random variables,
	\begin{equation} \label{eq:CovDef-MLmGWsSchroEqu2019}
	\mathbb E_\omega \big( \agl[\overline{f(\cdot,\omega) - \mathbb{E} (f(\cdot,\omega))},\varphi] \agl[f(\cdot,\omega) - \mathbb{E} (f(\cdot,\omega)),\psi] \big),
	\end{equation}
	can be understood as the covariance of $f$.
	Here $\mathbb E_\omega$ means to take expectation on the argument $\omega$. 
	Formula \eqref{eq:CovDef-MLmGWsSchroEqu2019} defines an operator $\mathfrak C_f$,
	\[
	\mathfrak C_f \colon \varphi \in \scrS(\Rn) \ \mapsto \ \mathfrak C_f (\varphi) \in \scrS'(\Rn),
	\]
	in a way that
	\(\mathfrak C_f (\varphi) \colon \psi \in \scrS(\Rn) \ \mapsto \ (\mathfrak C_f (\varphi))(\psi) \in \mathbb C\)
	where
	\[
	(\mathfrak C_f (\varphi))(\psi) := \mathbb E_\omega \big( \agl[\overline{f(\cdot,\omega) - \mathbb{E} (f(\cdot,\omega))},\varphi] \agl[f(\cdot,\omega) - \mathbb{E} (f(\cdot,\omega)),\psi] \big).
	\]
	The operator $\mathfrak C_f$ is called the covariance operator of $f$.
\begin{defn}[Migr field] \label{defn:migr-RSRe2020}
A generalized Gaussian random distribution $f$ on $\Rn$ is called microlocally isotropic with \emph{rough order} $-m$ and \emph{rough strength} $\mu(x)$ in a bounded domain $D$, if the following conditions hold:
\begin{enumerate}
	\item the expectation $\mathbb E (f)$ is in $\mathcal{C}_c^\infty(\Rn)$ with $\supp \mathbb E (f) \subset D$;
	
	\item $f$ is supported in $D$ a.s. (namely, almost surely);
	
	\item the covariance operator $\mathfrak C_f$ is a classical pseudodifferential operator of order $-m$;
	
	\item $\mathfrak C_f$ has a principal symbol of the form $\mu(x)|\xi|^{-m}$ with $\mu \in \mathcal{C}_c^\infty(\Rn;\R)$, $\supp \mu \subset D$ and $\mu(x) \geq 0$ for all $x \in \Rn$.
\end{enumerate}
We call a microlocally isotropic Gaussian random distribution as an \emph{migr} field.
\end{defn}

For the case where both the source and the potential are deterministic and $L^\infty$ functions with compact supports, the well-posedness of the direct problem of system \eqref{eq:1a-RSRe2020}--\eqref{eq:1c-RSRe2020} is known; see, e.g., \cite{colton2012inverse, eskin2011lectures, mclean2000strongly}.
Moreover, there holds the following asymptotic expansion of the outgoing radiating field $u^{sc}$ as $|x| \to +\infty$,
\begin{equation*} 
u^{sc}(x) = \frac{e^{\mathrm{i}k |x|}}{|x|^{(n-1)/2}} u^\infty(\hat x, k, d) + o(|x|^{-(n-1)/2}), \quad x \in \Rn.
\end{equation*}
$u^\infty(\hat x, k, d)$ is referred to as the far-field pattern, which encodes information of the potential and the source.
$\hat x:=x/|x|$ and $d$ in $u^\infty(\hat x, k, d)$ are unit vectors and they respectively stand for the observation direction and the impinging direction of the incident wave.
When $d = -\hat x$, $u^\infty(\hat x, k, -\hat x)$ is called the backscattering far-field pattern.
We shall see very soon that both the near-field $u^{sc}$ and the far-field $u^\infty$ can be used to achieve the recovery.

In \eqref{eq:1a-RSRe2020}, let us denote the source term as $f$ and the potential term as $q$.
In our study, both the source $f$ and the potential $q$ are assumed to be compactly supported. 
We shall treat \cite{LassasA, Lassas2008, caro2016inverse, LiHelinLiinverse2018, li2019inverse, li2020inverse, llm2018random,
	ma2020determining, ma2019determining}
in more details.
In \cite{LassasA, Lassas2008, caro2016inverse},
$q$ is assumed to be a migr field while $f$ is either zero or point a point source, i.e.~$\delta_y(x)$.
In \cite{LiHelinLiinverse2018, li2019inverse, li2020inverse}, $q$ is assumed to be zero and $f$ is assumed to be a migr field.
In \cite{llm2018random}, $q$ is assumed to be unknown and deterministic and $f$ is assumed to be a Gaussian white noise, while in \cite{ma2020determining, ma2019determining}, $q$ is assumed to be deterministic or migr type and $f$ is assumed to be a migr field.

In \cite{LassasA, Lassas2008} the authors considered the inverse scattering problem for the two-dimensional random Schr\"odinger system
$(-\Delta - k^2 - q(x,\omega)) u(x,k,\omega) = \delta_y(x)~(x \in \R^2)$ which is incited by point sources $u^{in}(x) = \frac {\mathrm{i}} {4} H_0^{(1)}(k|x-y|)$;
the $H_0^{(1)}$ is the Hankel function for the first kind, and the origin $y$ of this source is located in $U$.
The potential $q(x,\omega)$ is a migr field with compact support $D$ and $\overline U \cap \overline D = \emptyset$.
The main result is presented as follows (cf.~\cite[Theorem 7.1]{Lassas2008}).

\begin{thm} \label{thm:1-RSRe2020}
In \cite{LassasA, Lassas2008}, for $x,y \in U$ the limit
\[
R(x,y) = \lim_{K \to +\infty} \frac 1 {K-1} \int_1^K k^{2 + m} |u^{sc}(x,y,k,\omega)|^2 \dif k
\]
holds almost surely where
\[
R(x,x) := \frac 1 {2^{6 + m} \pi^2} \int_{\R^2} \frac {\mu_q(z)} {|x - z|^2} \dif z, \quad x \in U.
\]
and the $\mu_q$ is  the rough strength and $-m$ is the rough order of $q$.
\end{thm}

In \cite{caro2016inverse},
the authors considered
$(-\Delta - k^2 - q(x,\omega)) u(x,k,\omega) = 0, \ x \in \Rn$ with incident plane wave $u^{in}(x) = e^{ikd \cdot s}$.
The potential $q$ is assumed to be a migr field with compact support.
The main result (cf.~\cite[Corollary 4.4]{caro2016inverse}) is as follows.

\begin{thm} \label{thm:2-RSRe2020}
In \cite{caro2016inverse}, the limit
\[
\widehat{\mu}_q(2\tau \hat x) \simeq \lim_{K \to +\infty} \frac 1 K \int_K^{2K} k^m u^{\infty}(\hat x,-\hat x,k) \overline{u^{\infty}(\hat x,-\hat x,k + \tau)} \dif k, \quad \hat x \in \mathbb S^2, \ \tau > 0.
\]
holds almost surely.
\end{thm}
We note that the near-field data are used in \cite{LassasA, Lassas2008}, while in \cite{caro2016inverse},
the authors used far-field data.

Part of the results in \cite{LiHelinLiinverse2018} and \cite{li2019inverse, li2020inverse} are similar to each other and we only survey the first result in \cite{LiHelinLiinverse2018}.
In \cite{LiHelinLiinverse2018} the authors studied the Helmholtz equation
$(-\Delta - k^2) u(x) = f$ where $f$ is the source of migr type.
Note that the potential equals zero.
The main result (cf.~\cite[Theorem 3.9]{LiHelinLiinverse2018}) is similar to Theorem \ref{thm:1-RSRe2020}. 
\begin{thm} \label{thm:3-RSRe2020}
In \cite{LiHelinLiinverse2018}, the limit
\[
\int \frac {\mu_f(z)} {|x - z|} \dif z \simeq \lim_{K \to +\infty} \frac 1 {K-1} \int_1^K k^{1 + m} |u^{sc}(x,k,\omega)|^2 \dif k, \quad x \in U,
\]
holds almost surely.
\end{thm}

In \cite{llm2018random}, the authors considered direct and inverse scattering for \eqref{eq:1a-RSRe2020}--\eqref{eq:1c-RSRe2020}
with an unknown deterministic potential and a Gaussian noise source of the form $\sigma(x) \dot B_x(\omega)$, where $\sigma(x)$ is the variance and $\dot B_x(\omega)$ is the Gaussian white noise. 
The main (cf.~\cite[Lemma 4.3]{llm2018random}) result is 
\begin{thm} \label{thm:4-RSRe2020}
In \cite{llm2018random}, the identity
\[
\widehat{\sigma^2} (x) = 4\sqrt{2\pi} \lim_{j \to +\infty} \frac 1 {K_j} \int_{K_j}^{2K_j} \overline{u^\infty(\hat x,k,\omega)} \cdot u^\infty(\hat x,k+\tau,\omega) \dif{k}.
\]
holds almost surely.
\end{thm}

The paper \cite{ma2020determining} extended the work \cite{llm2018random} to the case where the source is a migr field $f$ with $\mu_f$ as its rough strength and $-m$ as its rough order.
For notational convenience, we shall use $\{K_j\} \in P(t)$ to signify a sequence $\{K_j\}_{j \in \mathbb{N}}$ satisfying $K_j \geq C j^t ~(j \in \mathbb{N})$ for some fixed constant $C > 0$. Throughout the rest of the paper, $\gamma$ stands for a fixed positive real number.
The main result (cf.~\cite[Theorem 4.3]{ma2020determining})
is presented below.

\begin{thm} \label{thm:5-RSRe2020}
In \cite{ma2020determining}, assume $2 < m < 3$ and let $m^* = \max \{2/3,\, (3-m)^{-1}/2 \}$. 
Assume that $\{K_j\} \in P(m^* + \gamma)$. 
Then  $\exists\, \Omega_0 \subset \Omega \colon \mathbb{P}(\Omega_0) = 0$, $\Omega_0$ depending only on $\{K_j\}_{j \in \mathbb{N}}$, such that for any $\omega \in \Omega \backslash \Omega_0$, there exists $S_\omega \subset \R^3 \colon |S_\omega| = 0$, it holds that for $\forall \tau \in \R_+$ and $\forall \hat x \in \mathbb{S}^2$ satisfying $\tau \hat x \in \R^3 \backslash S_\omega$,
\[
\widehat{\mu} (\tau \hat x) = 4\sqrt{2\pi} \lim_{j \to +\infty} \frac 1 {K_j} \int_{K_j}^{2K_j} k^m \overline{u^\infty(\hat x,k,\omega)} \cdot u^\infty(\hat x,k+\tau,\omega) \dif{k},
\]
holds for $\forall \tau \in \R_+$ and $\forall \hat x \in \mathbb{S}^2$ satisfying $\tau \hat x \in \R^3 \backslash S_\omega$.
\end{thm}

Then in \cite{ma2019determining} the authors further extended the work \cite{ma2020determining} to the case where both the potential $q$ and the source  $f$ are random of migr type.
The $f$ (resp.~$q$) is assumed to be supported in the domain $D_f$ (resp.~$D_q$).
In what follows, we assume that there is a positive distance between the convex hulls of the supports of $f$ and $q$, i.e.,
\begin{equation} \label{eq:fqSeparation-MLmGWsSchroEqu2019}
	\dist(\mathcal{CH}(D_f), \mathcal{CH}(D_q)) := \inf\{\, |x - y| \,;\, x \in \mathcal{CH}(D_f),\, y \in \mathcal{CH}(D_q) \,\} > 0,
\end{equation} 
where $\mathcal{CH}$ means taking the convex hull of a domain. 
Therefore, one can find a plane which separates $D_f$ and $D_q$. 
In order to simplify the exposition, we assume that $D_f$ and $D_q$ are convex domains and hence $\mathcal{CH}(D_f)=D_f$ and $\mathcal{CH}(D_q)=D_q$. 
Moreover, we let $\boldsymbol{n}$ denote the unit normal vector of the aforementioned plane that separates $D_f$ and $D_q$, pointing from the half-space containing $D_f$ into the half-space containing $D_q$.
Then the result of this work (cf.~\cite[Theorems 1.1 and 1.2]{ma2019determining}) is as follows.

\begin{thm} \label{thm:6-RSRe2020}
In \cite{ma2019determining}, suppose that $f$ and $q$ in system \eqref{eq:1a-RSRe2020}-\eqref{eq:1c-RSRe2020} are migr fields of order $-m_f$ and $-m_q$, respectively, satisfying
\[
2 < m_f < 4, \ m_f < 5m_q - 11.
\]
Assume that \eqref{eq:fqSeparation-MLmGWsSchroEqu2019} is satisfied and $\boldsymbol{n}$ is defined as above. 
Then, independent of $\mu_q$, $\mu_f$ can be uniquely recovered almost surely and the recovering formula of $\mu_f$ is given by
\begin{equation} \label{eq:Thm1Rec-MLmGWsSchroEqu2019}
	\displaystyle{ \widehat \mu_f(\tau \hat x)
		= \left\{\begin{aligned}
		& \lim_{K \to +\infty} \frac {4\sqrt{2\pi}} K \int_K^{2K} k^{m_f} \overline{u^\infty(\hat x,k,\omega)} u^\infty(\hat x,k+\tau,\omega) \dif k, \quad \hat{x} \cdot \boldsymbol{n} \geq 0,\medskip \\
		& \ \overline{\widehat \mu_f(-\tau \hat x)}, \quad \hat{x} \cdot \boldsymbol{n} < 0,
		\end{aligned}\right. }
	\end{equation}
	where $\tau \geq 0$ and $u^\infty(\hat x,k,\omega) \in \mathcal M_f(\omega) := \{\, u^\infty(\hat x, k, \omega) \,;\, \forall \hat{x} \in \mathbb{S}^2,\, \forall k \in \R_+\, \}$.

When $m_q < m_f$, $\mu_q$ can be uniquely recovered almost surely by the data set $\mathcal M_q(\omega)$ for a fixed $\omega \in \Omega$.
Moreover, the recovering formula is given by
	\begin{equation} \label{eq:Thm2Rec-MLmGWsSchroEqu2019}
	\displaystyle{ \widehat \mu_q(\tau \hat x)
		= \left\{\begin{aligned}
		& \lim_{K \to +\infty} \frac {4\sqrt{2\pi}} K \int_K^{2K} k^{m_q} \overline{u^\infty(\hat x,k,-\hat x,\omega)} u^\infty(\hat x,k\!+\! \tfrac \tau 2,-\hat x,\omega) \dif k, \ \hat{x} \cdot \boldsymbol{n} \geq 0,\medskip \\
		& \ \overline{\widehat \mu_f(-\tau \hat x)}, \quad \hat{x} \cdot \boldsymbol{n} < 0,
		\end{aligned}\right. }
\end{equation}
where $\tau \geq 0$ and $u^\infty(\hat x,k,-\hat x,\omega) \in \mathcal M_q(\omega) := \{\, u^\infty(\hat x, k, -\hat x, \omega) \,;\, \forall \hat{x} \in \mathbb{S}^2,\, \forall k \in \R_+\, \}$.
\end{thm}

\begin{rem}
In Theorem \ref{thm:6-RSRe2020}, the data sets $\mathcal M_f(\omega)$ and $\mathcal M_f(\omega)$ correspond to the case where the incident wave is passive and active, respectively.
Readers may refer to \cite[Section 1]{ma2019determining} for more details.
\end{rem}

Readers should note that the recovery formulae in Theorems \ref{thm:1-RSRe2020}--\ref{thm:6-RSRe2020} only use a single realization of the randomness; the terms on the left-hand-side are independent of the random sample $\omega$, while these on the right-hand-side are limits of terms depending on $\omega$.
This feature is also described as ``statistically stable'' in some literature.
The key ingredient to make this single-realization recovery possible is ergodicity; on the right-hand-side of these recoveries formulae in Theorems \ref{thm:1-RSRe2020}--\ref{thm:6-RSRe2020}, the probabilistic expectation operation are replaced by the average in the frequency variable and then taking to the infinity of the frequency variable.
Theorems \ref{thm:1-RSRe2020} and \ref{thm:3-RSRe2020} utilize near-field data to  achieve the recovery, while Theorem \ref{thm:2-RSRe2020} and \ref{thm:4-RSRe2020}--\ref{thm:6-RSRe2020} use far-field data.
Due to this difference, the corresponding techniques required in the proofs are also different.
We shall present these techniques separately in Sections \ref{sec:Matti-RSRe2020} and \ref{sec:Ma-RSRe2020}.

The rest of this paper is organized as follows. 
In Section \ref{sec:dp-RSRe2020}, we first give some preliminaries and present the well-posedness of the direct problems.
In Section \ref{sec:Matti-RSRe2020}, we give the sketch of the proofs in 
\cite{LassasA, Lassas2008, caro2016inverse, LiHelinLiinverse2018, li2019inverse, li2020inverse}.
Section \ref{sec:Ma-RSRe2020} is devoted to the details of the works \cite{ma2020determining, ma2019determining}.
We conclude the paper with some remarks and open problems in Section \ref{sec:conclusion-RSRe2020}.

\section{Preliminaries and the direct problems} \label{sec:dp-RSRe2020}

Due to the presence of the randomness, the regularity of the potential and/or the source may be too bad to fall into the scenarios of standard PDEs techniques.
In this section, we show some techniques used in reformulating the direct problems of  \eqref{eq:1a-RSRe2020}-\eqref{eq:1c-RSRe2020} in a proper sense.
Before that, we first 
present some preliminaries as well as some facts related to the migr field for the subsequent use.

\subsection{Preliminary and auxiliary results} \label{subsec:preli-SchroEqu2018}

For convenient reference and self-containedness, we first present some preliminary and auxiliary results.
In this paper, we mainly focus on the three-dimensional case. Nevertheless, some of the results derived also hold for higher dimensions and in those cases, we choose to present the results in the general dimension $n\geq 3$ since they might be useful in other studies.
Here we closely follow the paper \cite{ma2019determining}.

Throughout the paper, we write $\mathcal{L}(\mathcal A, \mathcal B)$ to denote the set of all the bounded linear mappings from a normed vector space $\mathcal A$ to a normed vector space $\mathcal B$. 
For any mapping $\mathcal K \in \mathcal{L}(\mathcal A, \mathcal B)$, we denote its operator norm as $\nrm[\mathcal{L}(\mathcal A, \mathcal B)]{\mathcal K}$. 
We also use $C$ and its variants, such as $C_D$, $C_{D,f}$, to denote some generic constants whose particular values may change line by line. 
For two quantities, we write $\mathcal{P}\lesssim \mathcal{Q}$ to signify $\mathcal{P}\leq C \mathcal{Q}$ and $\mathcal{P} \simeq \mathcal{Q}$ to signify $\widetilde{C}\mathcal{Q}\leq \mathcal{P} \leq C \mathcal{Q}$, for some generic positive constants $C$ and $\widetilde{C}$.
We write ``almost everywhere'' as~``a.e.''~and ``almost surely'' as~``a.s.''~for short. 
We use $|\mathcal S|$ to denote the Lebesgue measure of any Lebesgue-measurable set $\mathcal S$.

The Fourier transform and inverse Fourier transform of a function $\varphi$ are respectively defined as
\begin{align*} 
& \calF \varphi (\xi) = \widehat \varphi(\xi) := (2\pi)^{-n/2} \int e^{-{\textrm{i}}x\cdot \xi} \varphi(x) \dif x, \\
& \calF^{-1} \varphi (\xi) := (2\pi)^{-n/2} \int e^{{\textrm{i}}x\cdot \xi} \varphi(x) \dif x.
\end{align*}
Set
\[
\Phi(x,y) = \Phi_k(x,y) := \frac {e^{{\textrm{i}}k|x-y|}}{4\pi|x-y|}, \quad x\in\mathbb{R}^3\backslash\{y\}.
\]
$\Phi_k$ is the outgoing fundamental solution, centered at $y$, to the differential operator $-\Delta-k^2$. Define the resolvent operator $\Rk$,
\begin{equation} \label{eq:DefnRk-MLmGWsSchroEqu2019}
(\Rk \varphi)(x) := \int_{\R^3} \Phi_k(x,y) \varphi(y) \dif{y}, \quad x \in \R^3,
\end{equation}
where $\varphi$ can be any measurable function on $\mathbb{R}^3$ as long as \eqref{eq:DefnRk-MLmGWsSchroEqu2019} is well-defined for almost all $x$ in $\R^3$. 

Write $\agl[x] := (1+|x|^2)^{1/2}$ for $x \in \Rn$, $n\geq 1$. We introduce the following weighted $L^p$-norm and the corresponding function space over $\Rn$ for any $\delta \in \R$,
\begin{equation} \label{eq:WetdSpace-MLmGWsSchroEqu2019}
\begin{aligned}
\nrm[L_\delta^p(\Rn)]{\varphi} := \ & \nrm[L^p(\Rn)]{\agl[\cdot]^{\delta} \varphi(\cdot)} = \big( \int_{\Rn} \agl[x]^{p\delta} |\varphi|^p \dif{x} \big)^{\frac 1 p}, \\
L_\delta^p(\Rn) := \ & \{\, \varphi \in L_{loc}^1(\Rn) \,;\, \nrm[L_\delta^p(\Rn)]{\varphi} < +\infty \,\}.
\end{aligned}
\end{equation}
We also define $L_\delta^p(S)$ for any subset $S$ in $\Rn$ by replacing $\Rn$ in \eqref{eq:WetdSpace-MLmGWsSchroEqu2019} with $S$. 
In what follows, we may write $L_\delta^2(\R^3)$ as $L_\delta^2$ for short without ambiguities.
Let $I$ be the identity operator and define 
\begin{equation*}
\nrm[H_\delta^{s,p}(\Rn)]{f} := \nrm[L_\delta^p(\Rn)]{(I-\Delta)^{s/2} f}, \ H_\delta^{s,p}(\Rn) = \{ f \in \scrS'(\Rn); \nrm[H_\delta^{s,p}(\Rn)]{f} < +\infty\},
\end{equation*}
where $\scrS'(\Rn)$ stands for the dual space of the Schwartz space $\scrS(\Rn)$.
The space $H_\delta^{s,2}(\Rn)$ is abbreviated as $H_\delta^s(\Rn)$, and $H_0^{s,p}(\Rn)$ is abbreviated as $H^{s,p}(\Rn)$.
It can be verified that
\begin{equation} \label{eq:normEquiv-MLmGWsSchroEqu2019}
\nrm[H_\delta^s(\Rn)]{f} = \nrm[H^\delta(\Rn)]{\agl[\cdot]^s \widehat f(\cdot)}.
\end{equation}

Let $m \in (-\infty,+\infty)$. We define $S^m$ to be the set of all functions $\sigma(x,\xi) \in C^{\infty}(\Rn,\Rn;\mathbb C)$ such that for any two multi-indices $\alpha$ and $\beta$, there is a positive constant $C_{\alpha, \beta}$, depending on $\alpha$ and $\beta$ only, for which
\[
\big| (D_{x}^{\alpha}D_{\xi}^{\beta}\sigma)(x,\xi) \big| \leq C_{\alpha, \beta}(1+|\xi|)^{m-|\beta|}, \quad \forall x, \xi \in \Rn.
\]
We call any function $\sigma$ in $\bigcup_{m \in \R} S^m$ a \emph{symbol}.
A \emph{principal symbol} of $\sigma$ is an equivalent class $[\sigma] = \{ \tilde \sigma \in S^m \,;\, \sigma - \tilde \sigma \in S^{m-1} \}$. 
In what follows, we may use one representative $\tilde \sigma$ in $[\sigma]$ to represent the equivalent class $[\sigma]$.
Let $\sigma$ be a symbol. Then the \emph{pseudo-differential operator} $T$, defined on $\scrS(\Rn)$ and associated with $\sigma$, is defined by
\begin{align*}
	(T_{\sigma}\varphi)(x)
	& := (2\pi)^{-n/2} \int_{\Rn} e^{{\textrm{i}}x \cdot \xi} \sigma(x,\xi) \hat{\varphi}(\xi) \dif{\xi} \\
	& \ = (2\pi)^{-n} \iint_{\Rn \times \Rn} e^{{\textrm{i}}(x-y) \cdot \xi} \sigma(x,\xi) \varphi(y) \dif y \dif{\xi}, \quad \forall \varphi \in \scrS(\Rn).
\end{align*}

\smallskip

Recall Definition \ref{defn:migr-RSRe2020}.
Lemma \ref{lem:migrRegu-MLmGWsSchroEqu2019} below shows how the rough order of a migr field is related to its Sobolev regularity.

\begin{lem} \label{lem:migrRegu-MLmGWsSchroEqu2019}
	Let $h$ be an migr distribution of rough order $-m$ in $D_h$. Then, $h \in H^{-s,p}(\Rn)$ almost surely for any $1 < p < +\infty$ and $s > (n-m)/2$.
\end{lem}
\begin{proof}[Proof of Lemma \ref{lem:migrRegu-MLmGWsSchroEqu2019}]

See Proposition 2.4 in \cite{caro2016inverse}.
\end{proof}

By the Schwartz kernel theorem (see Theorem 5.2.1 in \cite{hormander1985analysisI}), there exists a kernel $K_h(x,y)$ with $\supp K_h \subset D_h \times D_h$ such that
\begin{equation} \label{eq:CK-MLmGWsSchroEqu2019}
(\mathfrak C_h \varphi)(\psi) 
= \mathbb E_\omega ( \agl[\overline{h(\cdot,\omega)},\varphi] \agl[h(\cdot,\omega),\psi]) 
= \iint K_h(x,y) \varphi(x) \psi(y) \dif x \dif y,
\end{equation}
for all $\varphi$, $\psi \in \scrS(\Rn)$.
It is easy to verify that $K_h(x,y) = \overline{K_h(y,x)}$.
Denote the symbol of $\mathfrak C_h$ as $c_h$, then it can be verified \cite{caro2016inverse} that the equalities
\begin{subequations} \label{eq:KandSymbol-MLmGWsSchroEqu2019}
	\begin{numcases}{}
	K_h(x,y) = (2\pi)^{-n} \int e^{{\textrm{i}}(x-y) \cdot \xi} c_h(x,\xi) \dif \xi, \label{eq:KtoSymbol-MLmGWsSchroEqu2019} \\
	c_h(x,\xi) = \int e^{-{\textrm{i}}\xi\cdot(x-y)} K_h(x,y) \dif x, \label{eq:SymboltoK-MLmGWsSchroEqu2019}
	\end{numcases}
\end{subequations}
hold in the distributional sense, and the integrals in \eqref{eq:KandSymbol-MLmGWsSchroEqu2019} shall be understood as oscillatory integrals.
{Despite} the fact that $h$ usually is not a function, {intuitively speaking, however,} it is helpful to keep in mind the following correspondence,
\[
K_h(x,y) \sim \mathbb E_\omega \big( \overline{h(x,\omega)} h(y,\omega) \big).
\]

\subsection{Some techniques related to the direct problem} \label{subsec:DP-RSRe2020}

On way to study the direct problem of \eqref{eq:1a-RSRe2020}-\eqref{eq:1c-RSRe2020} is to transform it into the Lippmann-Schwinger equation, and then use the Bonn expansion to define the solution.
To that end, the estimate of the operator norm of the resolvent $\Rk$ is crucial.
Among different types of the estimates in the literature, one of them is known as Agmon's estimate (cf.~\cite[\S 29]{eskin2011lectures}).
Reformulating \eqref{eq:1a-RSRe2020} into the Lippmann-Schwinger equation formally (cf. \cite{colton2012inverse}), we obtain
\begin{equation*} 
(I - \Rk q) u^{sc} = \alpha \Rk q u^{in} - \Rk f.
\end{equation*}

We demonstrate two lemmas dealing with the lack of regularity when utilizing Agmon's estimates.
Lemma \ref{lemma:RkfBdd-RSRe2020} (cf.~\cite[Lemma 2.2]{ma2020determining}) shows the resolvent can take a migr field as an input without any trouble,
while Lemma \ref{lem:RkBounded-RSRe2020} (cf.~\cite[Theorem 2.1]{ma2019determining}) gives a variation of Agmon's estimate to fit our own problem settings.

\begin{lem} \label{lemma:RkfBdd-RSRe2020}
Assume $f$ is a migr field with rough order $-m$ and $\supp f \subset D_f$ almost surely, then we have $\Rk f \in L_{-1/2-\epsilon}^2$ for any $\epsilon > 0$ almost surely.
\end{lem}
	
	\begin{proof}
We split $\Rk f$ into two parts, $\Rk ( \mathbb{E}f)$ and $\Rk (f - \mathbb{E}f)$.
\cite[Lemma 2.1]{llm2018random} gives $\Rk ( \mathbb{E}f) \in L_{-1/2-\epsilon}^2$.	
For $\Rk (f - \mathbb{E}f)$, by using \eqref{eq:CK-MLmGWsSchroEqu2019}, \eqref{eq:KandSymbol-MLmGWsSchroEqu2019} and \eqref{eq:DefnRk-MLmGWsSchroEqu2019}, one can compute
\begin{align}
& \mathbb{E} ( \nrm[L_{-1/2-\epsilon}^2]{\Rk (f - \mathbb{E}f)(\cdot,\omega)}^2 ) \nonumber\\
= & \int_{\R^3} \agl[x]^{-1-2\epsilon} \mathbb{E} ( \agl[\overline {f - \mathbb{E}f}, \Phi_{-k,x}] \agl[f - \mathbb{E}f, \Phi_{k,x}] ) \dif{x} = \int_{\R^3} \agl[x]^{-1-2\epsilon} \agl[\mathfrak C_f \Phi_{-k,x}, \Phi_{k,x}] \dif{x} \nonumber\\
{\color{black}\simeq }& \int \agl[x]^{-1-2\epsilon} \int_{D_f} \big( \int_{D_f} \frac{\mathcal I(y,z) e^{-ik|x-z|}}{|x-z| \cdot |y-z|^2} \dif{z} \big) \cdot \frac{e^{ik|x-y|}}{|x-y|} \dif{y} \dif{x}, \label{eq:gkInter1-RSRe2020}
\end{align}
where $c_f(y,\xi)$ is the symbol of the covariance operator $\mathfrak C_f$ and
\[
\mathcal I(y,z) := \int_{\R^3} |y-z|^2 e^{i(y-z) \cdot \xi} c_f(y,\xi) \dif{\xi}.
\]
When $y = z$, we know $\mathcal I(y,z) = 0$ because the integrand is zero. Thanks to the condition $m > 2$, when $y \neq z$ we have
\begin{align}
	|\mathcal I(y,z)| & 
	= \big| \sum_{j=1}^3 \int_{\R^3} e^{i(y-z) \cdot \xi} (\partial_{\xi_j}^2 c_f)(y,\xi) \dif{\xi} \big| 
	\sum_{j=1}^3 \int_{\R^3} C_j \agl[\xi]^{-m-2} \dif{\xi} \leq C_0 < +\infty, \label{eq:gkInter2-MLSchroEqu2018}
\end{align}
for some constant $C_0$ independent of $y$ and $z$.
Note that if $D_f$ is bounded, then for $j=1,2$ we have
\begin{equation} \label{eq:gkInter3-MLSchroEqu2018}
	\int_{D_f} |x-y|^{-j} \dif{y} \leq C_{f,j} \agl[x]^{-j}, \quad \forall x \in \R^3,
\end{equation}
for some constant $C_{f,j}$ depending only on $f, j$ and the dimension.
The notation  $\agl[x]$ in \eqref{eq:gkInter3-MLSchroEqu2018} stands for $(1 + |x|^2)^{1/2}$ and readers may note the difference between the $\agl[\cdot]$ and the $\agl[\cdot,\cdot]$ appeared in \eqref{eq:DefnRk-MLmGWsSchroEqu2019}.
With the help of \eqref{eq:gkInter2-MLSchroEqu2018} and \eqref{eq:gkInter3-MLSchroEqu2018} and H\"older's inequality, we can continue \eqref{eq:gkInter1-RSRe2020} as
\begin{align*}
& \mathbb{E} ( \nrm[L_{-1/2-\epsilon}^2]{\Rk (f - \mathbb{E}f)(\cdot,\omega)}^2 ) \nonumber\\
\lesssim & \int \agl[x]^{-1-2\epsilon} \big( \iint_{D_f \times D_f} (|x-z| \cdot |y-z|^2 \cdot |x-y|)^{-1} \dif{z} \dif{y} \big) \dif{x} \nonumber\\
\leq & \int \agl[x]^{-1-2\epsilon} C_f \agl[x]^{-2} \dif{x}  \leq C_f < +\infty,
\end{align*}
which gives
\begin{equation} \label{eq:RkfBounded-MLSchroEqu2018}
\mathbb{E} ( \nrm[L_{-1/2-\epsilon}^2]{\Rk (f - \mathbb{E}f)(\cdot,\omega)}^2 ) \leq C_f < +\infty.
\end{equation}
By using the H\"older inequality applied to the probability measure, we obtain from \eqref{eq:RkfBounded-MLSchroEqu2018} that
\begin{equation} \label{eq:RkfBoundedCD-MLSchroEqu2018}
\mathbb{E} \nrm[L_{-1/2-\epsilon}^2]{\Rk (f - \mathbb{E}f)} \leq [ \mathbb{E} ( \nrm[L_{-1/2-\epsilon}^2]{\Rk (f - \mathbb{E}f)}^2 ) ]^{1/2} \leq C_f^{1/2} < +\infty,
\end{equation}
for some constant $C_f$ independent of $k$. The formula (\ref{eq:RkfBoundedCD-MLSchroEqu2018}) gives that
\(
\Rk (f - \mathbb{E} f) \in L_{-1/2-\epsilon}^2
\)
almost surely, and hence
\(
\Rk f \in L_{-1/2-\epsilon}^2
\)
almost surely.
	
The proof is complete. 
	\end{proof}

\begin{lem} \label{lem:RkBounded-RSRe2020}
	For any $0 < s < 1/2$ and $\epsilon > 0$, when $k > 2$,
	$$\nrm[H_{-1/ 2 - \epsilon}^s(\R^3)]{\Rk \varphi} \leq C_{\epsilon,s} k^{-(1 - 2s)} \nrm[H_{1/2 + \epsilon}^{-s}(\R^3)]{\varphi}, \quad \varphi \in H_{1/2 + \epsilon}^{-s}(\R^3).$$
\end{lem}


\begin{proof}
In this proof we adopt the concept of {\it Limiting absorption principle} to first show desired results on a family of operator $\mathcal{R}_{k,\tau}$ controlled by a parameter $\tau$, and then show that $\mathcal{R}_{k,\tau}$ converges in a proper sense as $\tau$ approaches zero.
We sketch out the key steps in the proof and readers may refer to the proof of \cite[Theorem 2.1]{ma2019determining} for complete details.
	
	Define an operator	\begin{equation}\label{eq:opnnna1}
\mathcal{R}_{k,\tau} \varphi(x) := (2\pi)^{-3/2} \int_{\R^3} e^{{\textrm{i}}x\cdot \xi} \frac {\hat \varphi(\xi)} {|\xi|^2 - k^2 - \textrm{i}\tau} \dif \xi,
	\end{equation}
	where $\tau\in\mathbb{R}_+$ and $\textrm{i} := \sqrt{-1}$. 
Fix a function $\chi$ satisfying
	\begin{equation} \label{eq:cutoffFunc-MLmGWsSchroEqu2019}
	\left\{\begin{aligned} 
	& \chi \in C_c^\infty(\Rn),\, 0 \leq \chi \leq 1, \\
	& \chi(x) = 1 \mbox{ when } |x| \leq 1, \\
	& \chi(x) = 0 \mbox{ when } |x| \geq 2.
	\end{aligned}\right.
	\end{equation}
	Write $\mathfrak{R} \psi(x) := \psi(-x)$. 
	We have
	\begin{align}
	& \ (\mathcal{R}_{k,\tau} \varphi,\psi)_{L^2(\R^3)} \nonumber\\
	= \ & \ \int_{\R^3} \mathcal{R}_{k,\tau} \varphi (x) \overline{\psi(x)} \dif x = \int_{\R^3} \calF\{\mathcal{R}_{k,\tau} \varphi\} (\xi) \cdot \calF\{\mathfrak{R} \overline \psi\}(\xi) \dif \xi \nonumber\\
	= \ & \ \int_0^\infty \frac {(1 - \chi^2(r - k))} {r^2 - k^2 - {\textrm{i}}\tau} \dif r \cdot \int_{|\xi| = r} \hat \varphi (\xi) \cdot \widehat{\mathfrak{R} \overline \psi}(\xi) \dif{S(\xi)} \nonumber\\
	& \ + \int_0^\infty \frac {\agl[r]^{1/p}\, r^2 \chi^2(r - k)} {r^2 - k^2 - {\textrm{i}}\tau} \dif r \times \int_{\mathbb{S}^2} [\agl[k]^{\frac {-1} {2p}} \hat \varphi (k\omega)] [\agl[k]^{\frac {-1} {2p}} \widehat{\mathfrak{R} \overline \psi}(k\omega)] \dif{S(\omega)} \nonumber\\
	& \ + \int_0^\infty \frac {\agl[r]^{1/p}\, r^2 \chi^2(r - k)} {r^2 - k^2 - {\textrm{i}}\tau} \dif r \cdot \int_{\mathbb{S}^2} \{ [\agl[r]^{\frac {-1} {2p}} \hat \varphi (r\omega)] [\agl[r]^{\frac {-1} {2p}} \widehat{\mathfrak{R} \overline \psi}(r\omega)] \nonumber\\
	& \hspace*{33ex} - [\agl[k]^{\frac {-1} {2p}} \hat \varphi (k\omega)] [\agl[k]^{\frac {-1} {2p}} \widehat{\mathfrak{R} \overline \psi}(k\omega)] \} \dif{S(\omega)} \nonumber\\
	=: & \ I_1(\tau) + I_2(\tau) + I_3(\tau). \label{eq:RkBoundedI123-MLmGWsSchroEqu2019}
	\end{align}
Here we divide $(\mathcal{R}_{k,\tau} \varphi,\psi)_{L^2(\R^3)}$ into three parts in order to deal with the singularity happened in the integral when $|\xi|$ is close to $k$.
The integral in  $I_1$ has avoided this singularity by the cutoff function $\chi$.
The singularity in $I_2$ is only contained in the integration w.r.t.~$r$, and it can be shown that by using Cauchy's integral theorem and choosing a proper integral path w.r.t.~$r$, the norm of the denominator $\tau^2 - k^2 - i\tau$ can always be bounded below by $k$, e.g.~$|\tau^2 - k^2 - i\tau| \gtrsim k$.
The singularity in  $I_3$ is compensated by the difference $[\cdots]$ inside the integration $\int_{\mathbb S^2} [\cdots] \dif{S(\omega)}$.
In the following, we only show how to deal with $I_2$.

Now we estimate $I_1(\tau)$. By Young's inequality $ab \leq a^p/p + b^q/q$, for $a,b > 0,\, p,q > 1,\, 1/p + 1/q = 1$ we have
\begin{equation} \label{eq:YoungIneq-MLmGWsSchroEqu2019}
	(p^{1/p} q^{1/q}) a^{1/p} b^{1/q} \leq a + b.
\end{equation}
Note that $|r - k| > 1$ in the support of the function $1 - \chi^2(r - k)$ and $|\widehat{\mathfrak{R} \overline \psi}(\xi)| = |\hat \psi(\xi)|$, one can compute
\begin{align}
	|I_1(\tau)|
	& \leq \int_0^\infty \frac {1 - \chi^2(r - k)} {1 \cdot p^{1/p} q^{1/q} (r+1)^{1/p} (k-1)^{1/q}} \dif r \cdot \int_{|\xi| = r} |\hat \varphi (\xi)| \cdot |\hat \psi(\xi)| \dif{S(\xi)} \quad (\text{by } \eqref{eq:YoungIneq-MLmGWsSchroEqu2019}) \nonumber\\
	& \leq C_p k^{1/p - 1} \nrm[H_\delta^{-1/(2p)}(\R^3)]{\varphi} \nrm[H_\delta^{-1/(2p)}(\R^3)]{\psi}, \label{eq:RkBoundedI1-MLmGWsSchroEqu2019}
\end{align}
where $1 < p < +\infty$ and $\delta > 0$ and the $C_p$ is independent of $\tau$.

\smallskip

We next estimate $I_2(\tau)$. One has
\begin{align}
	I_2(\tau)
	& = \int_{\mathbb{S}^2} [\agl[k]^{\frac {-1} {2p}} \hat \varphi (k\omega)] [\agl[k]^{\frac {-1} {2p}} \widehat{\mathfrak{R} \overline \psi}(k\omega)] \int_0^\infty \frac {\agl[r]^{\frac 1 p} r^2 \chi^2(r - k) \dif r} {r^2 - k^2 - {\textrm{i}}\tau} \dif{S(\omega)}. \label{eq:RkBoundedI2Inter1-MLmGWsSchroEqu2019}
	\end{align}
It can be shown that, by choosing a fixed $\tau_0 \in (0,1)$ carefully, we can show that the denominator $p_\tau(r) := r^2 - k^2 - \textrm{i} \tau$ could satisfy	\begin{equation} \label{eq:p-MLmGWsSchroEqu2019}
	|p_\tau(r)| \geq \tau_0 k \text{~and~} |r| \lesssim k, \ \forall r \in \{ r ; 2 \geq |r - k| \geq \tau_0 \} \cup \Gamma_{k,\tau_0},\, \forall \tau \in (0,\tau_0),
	\end{equation}
	where $\Gamma_{k,\tau_0} := \{r \in \mathbb C ; |r - k| = \tau_0, \Im r \leq 0 \}$.
	It is obvious that the purpose of \eqref{eq:p-MLmGWsSchroEqu2019} is to use Cauchy's integral theorem.
By combining \eqref{eq:p-MLmGWsSchroEqu2019} with Cauchy's integral theorem, we can continue	\eqref{eq:RkBoundedI2Inter1-MLmGWsSchroEqu2019} as
	\begin{align}
	|I_2(\tau)|
	& \leq \int_{|\xi|=k} \agl[\xi]^{\frac {-1} {2p}} |\hat \varphi (\xi)| \cdot \agl[\xi]^{\frac {-1} {2p}} |\hat \psi(\xi)| \big( \int_{\{ r \in \R_+ \,;\, 2 \geq |r - k| \geq \tau_0 \}} \frac {\agl[r]^{\frac 1 p} (r/k)^2} {\tau_0 k} \dif r \big) \dif{S(\xi)} \nonumber\\
	& \ \ \ + \int_{|\xi|=k} \agl[\xi]^{\frac {-1} {2p}} |\hat \varphi (\xi)| \cdot \agl[\xi]^{\frac {-1} {2p}} |\hat \psi(\xi)| \big( \int_{\Gamma_{k,\tau_0}} \frac {(1+|r|^2)^{\frac 1 {2p}} (|r|/k)^2} {\tau_0 k} \dif r \big) \dif{S(\xi)} \nonumber \\ 
& \leq C_{\tau_0} \int_{|\xi|=k} \agl[\xi]^{\frac {-1} {2p}} |\hat \varphi (\xi)| \agl[\xi]^{\frac {-1} {2p}} |\hat \psi(\xi)| \big( \int_{\Gamma_{k,\tau_0} \cup \{ r \in \R_+; 2 \geq |r - k| \geq \tau_0 \}} \frac {\agl[k]^{1/p}} {\tau_0 k} \dif r \big) \dif{S(\xi)} \nonumber\\
	& \ \ \ + C_{\tau_0} \int_{|\xi|=k} \agl[\xi]^{\frac {-1} {2p}} |\hat \varphi (\xi)| \agl[\xi]^{\frac {-1} {2p}} |\hat \psi(\xi)| \big( \int_{\Gamma_{k,\tau_0}} \frac {\agl[k]^{1/p}} {\tau_0 k} \dif r \big) \dif{S(\xi)} \nonumber\\
	& \leq C_{\tau_0} k^{1/p - 1} \big( \int_{|\xi|=k} |\agl[\xi]^{\frac {-1} {2p}} \hat h(\xi)|^2 \dif{S(\xi)} \big)^{\frac 1 2} \big( \int_{|\xi|=k} |\agl[\xi]^{\frac {-1} {2p}} \hat \psi(\xi)|^2 \dif{S(\xi)} \big)^{\frac 1 2} \nonumber\\
	& \leq C_{\tau_0,\epsilon} k^{1/p - 1} \nrm[H_{1/2 + \epsilon}^{-1/(2p)}(\R^3)]{\varphi} \nrm[H_{1/2 + \epsilon}^{-1/(2p)}(\R^3)]{\psi}, \label{eq:RkBoundedI2-MLmGWsSchroEqu2019}
	\end{align}
	where the constant $C_{\tau_0,\epsilon}$ is independent of $\tau$.
Here, in deriving the last inequality in \eqref{eq:RkBoundedI2-MLmGWsSchroEqu2019}, we have made use of \eqref{eq:normEquiv-MLmGWsSchroEqu2019}.

	\smallskip

	Finally, we estimate $I_3(\tau)$. Denote
	$\mathbb F(r\omega) = \mathbb F_r(\omega) := \agl[r]^{-1/(2p)} \hat \varphi(r\omega)$ and $\mathbb G(r\omega) = \mathbb G_r(\omega) := \agl[r]^{-1/(2p)} \widehat{\mathfrak R \bar \psi}(r\omega)$. One can compute
	\begin{align}
	|I_3(\tau)|
	& \leq \int_0^\infty \frac {\agl[r]^{1/p} \chi^2(r - k)} {|r^2 - k^2|} \cdot \nrm[L^2(\mathbb S_r^2)]{\mathbb F_r} \cdot \big( r^2 \int_{\mathbb{S}^2} |\mathbb G_r - \mathbb G_k|^2 \dif{S(\omega)} \big)^{\frac 1 2} \dif r \nonumber\\
	& \ \ \ + \int_0^\infty \frac {\agl[r]^{1/p} \chi^2(r - k)} {|r^2 - k^2|} \cdot \big( r^2 \int_{\mathbb{S}^2} |\mathbb F_r - \mathbb F_k|^2 \dif{S(\omega)} \big)^{\frac 1 2} \cdot \big(\frac r k \big)^2 \nrm[L^2(\mathbb S_k^2)]{\mathbb G_k} \dif r, \label{eq:RkBoundedI3.1-MLmGWsSchroEqu2019}
	\end{align}
	where $\mathbb S_r^2$ signifies the central sphere of radius $r$.
	Combining \cite[Remark 13.1 and (13.28)]{eskin2011lectures} and \eqref{eq:normEquiv-MLmGWsSchroEqu2019} and \eqref{eq:YoungIneq-MLmGWsSchroEqu2019}, we can continue \eqref{eq:RkBoundedI3.1-MLmGWsSchroEqu2019} as
	\begin{align}
	|I_3(\tau)|
	& \leq C_{\alpha,\epsilon} \int_0^\infty \frac {\agl[r]^{1/p} \chi^2(r - k)} {|r - k| (r+k)} \cdot \nrm[H^{1/2+\epsilon}(\R^3)]{\mathbb F} \cdot |r - k|^\alpha \cdot \nrm[H^{1/2+\epsilon}(\R^3)]{\mathbb G} \dif r \nonumber\\
	& \leq C_{\alpha,\epsilon,p} \int_0^\infty \frac {\agl[r]^{1/p} \chi^2(r - k)} {|r - k|^{1-\alpha} (r+1)^{1/p} (k-1)^{1-1/p}} \dif r \cdot \nrm[H^{1/2+\epsilon}(\R^3)]{\mathbb F} \nrm[H^{1/2+\epsilon}(\R^3)]{\mathbb G} \nonumber\\
	& \leq C_{\alpha,\epsilon,p} k^{1/p-1} \nrm[H_{1/2+\epsilon}^{-1/(2p)}(\R^3)]{\varphi} \cdot \nrm[H_{1/2+\epsilon}^{-1/(2p)}(\R^3)]{\psi}, \label{eq:RkBoundedI3-MLmGWsSchroEqu2019}
	\end{align}
	where the $\epsilon$ can be any positive real number and  the $\alpha$ satisfies $0 < \alpha < \epsilon$, and the constant $C_{\alpha,\epsilon,p}$ is independent of $\tau$.
	
	Combining \eqref{eq:RkBoundedI123-MLmGWsSchroEqu2019}, \eqref{eq:RkBoundedI1-MLmGWsSchroEqu2019}, \eqref{eq:RkBoundedI2-MLmGWsSchroEqu2019} and \eqref{eq:RkBoundedI3-MLmGWsSchroEqu2019}, we arrive at
	\begin{equation*}
	|(\mathcal{R}_{k,\tau} \varphi,\psi)_{L^2(\R^3)}| \leq |I_1(\tau)| + |I_2(\tau)| + |I_3(\tau)| 
	\leq C k^{1/p - 1} \nrm[H_{1/2 + \epsilon}^{-1/(2p)}(\R^3)]{\varphi} \nrm[H_{1/2 + \epsilon}^{-1/(2p)}(\R^3)]{\psi},
	\end{equation*}
	which implies that
	\begin{equation} \label{eq:RkeBoundedf-MLmGWsSchroEqu2019}
	\nrm[H_{-1/2 - \epsilon}^{1/(2p)}(\R^3)]{\mathcal{R}_{k,\tau} \varphi} 
	\leq C k^{1/p - 1} \nrm[H_{1/2 + \epsilon}^{-1/(2p)}(\R^3)]{\varphi}
	\end{equation}
	for some constant $C$ independent of $\tau$.

	\medskip

	Next we investigate the limiting case $\lim\limits_{\tau \to 0^+} \mathcal{R}_{k,\tau} \varphi$. 
Following similar steps when dealing 	with $I_1$, $I_2$ and $I_3$, it can be shown that for any $\tilde \tau > 0$, we have
\[
|I_j(\tau_1) - I_j(\tau_2)| \leq \tilde\tau^\beta k^{1/p - 1} \nrm[H_{1/2 + \epsilon}^{-1/(2p)}(\R^3)]{\varphi} \nrm[H_{1/2 + \epsilon}^{-1/(2p)}(\R^3)]{\psi},
\quad (j = 1,2,3)
\]
holds for $\forall \tau_1, \tau_2 \in (0,\tilde \tau)$.
Therefore, we can conclude	\begin{equation*} 
	\nrm[H_{-1/2 - \epsilon}^{-1/(2p)}(\R^3)]{\mathcal{R}_{k,\tau_1} \varphi - \mathcal{R}_{k,\tau_2} \varphi} 
	\lesssim \tilde \tau \nrm[H_{1/2 + \epsilon}^{-1/(2p)}(\R^3)]{\varphi}, \quad \forall \tau_1,\tau_2 \in (0,\tilde \tau),
	\end{equation*}
	and thus $\mathcal{R}_{k,\tilde \tau} \varphi$ converges and
	\begin{equation} \label{eq:LimAbsPrin2-MLmGWsSchroEqu2019}
	\lim_{\tilde \tau \to 0^+} \mathcal{R}_{k,\tilde \tau} \varphi = \Rk \varphi \quad\text{ in }\quad H_{-1/2 - \epsilon}^{1/(2p)}(\R^3).
	\end{equation}
	Hence from \eqref{eq:RkeBoundedf-MLmGWsSchroEqu2019} and \eqref{eq:LimAbsPrin2-MLmGWsSchroEqu2019} we conclude that
	$$\nrm[H_{-1/2 - \epsilon}^{1/(2p)}(\R^3)]{\Rk \varphi} \leq C_{\epsilon,p} k^{-(1 - 1/p)} \nrm[H_{1/2 + \epsilon}^{-1/(2p)}(\R^3)]{\varphi}$$
	holds for any $1 < p < +\infty$ and any $\epsilon > 0$. 
	
	The proof is complete.
\end{proof}

With the help of
Lemmas \ref{lemma:RkfBdd-RSRe2020} and \ref{lem:RkBounded-RSRe2020}, the direct problems can be reformulated.
Readers may refer to \cite[Theorem 2.1]{ma2020determining}, \cite[Theorem 2.3]{ma2019determining}, \cite[Theorem 4.3]{Lassas2008}, \cite[Theorem 3.3]{LiHelinLiinverse2018}, and \cite[Theorem 3.3]{li2019inverse} as examples of how to formulate the direct problems, and we omit these details here.

\section{Recovery by near-field data} \label{sec:Matti-RSRe2020}

In this section we consider key steps in the works
\cite{LassasA, Lassas2008, caro2016inverse, LiHelinLiinverse2018, li2019inverse, li2020inverse}.
Lemma \ref{lem:kka-RSRe2020} is crucial in the key steps of the works, and its proof relies on Lemmas \ref{lem:intab-RSRe2020} and \ref{lem:kka-RSRe2020}.
We shall first investigate these useful lemmas.

\subsection{Useful lemmas} \label{sec:MUL-RSRe2020}

Lemma \ref{lem:intab-RSRe2020} is a standard result in the field of oscillatory integral and microlocal analysis.

\begin{lem} \label{lem:intab-RSRe2020}
Assume $\alpha$ and $\beta$ are multi-indexes,
then the following identities hold in the oscillatory integral sense,
\begin{align} 
\int_{\R_x^n \times \R_\xi^n} e^{ix\cdot \xi} \dif x \dif \xi & = (2\pi)^{n}, \label{eq:xxiInt-SuppMater} \\
\int_{\R_x^n \times \R_\xi^n} e^{ix\cdot \xi} x^\alpha \xi^\beta \dif x \dif \xi & = (2\pi)^{n} \alpha! \delta^{\alpha \beta}. \label{eq:xxiabInt-SuppMater}
\end{align}
\end{lem}

\begin{proof}
	The integral in \eqref{eq:xxiInt-SuppMater} should be understood as oscillatory integral. 
	Fix a cut off function $\chi \in C_c^\infty(\Rn)$ with $\chi(0) = 1$, we can compute
	\begin{align}
	\int_{\R_x^n \times \R_\xi^n} e^{ix\cdot \xi} \dif x \dif \xi
	& = \lim_{\epsilon \to 0^+} \int e^{ix\cdot \xi} \chi(\epsilon x) \chi(\epsilon \xi) \dif x \dif \xi
	= (2\pi)^{n/2} \lim_{\epsilon \to 0^+} \int \chi(\epsilon^2 \xi) \widehat\chi(-\xi) \dif \xi. \label{eq:xxiInt1-SuppMater}
	\end{align}
	Denote $M = \sup_{\Rn} \chi$.
	We have $|\chi(\epsilon^2 \xi)| \leq M < \infty$.
	Note that $\chi \in C_c^\infty(\Rn)$, so $\widehat \chi$ is rapidly decaying, thus $\widehat\chi(-\xi)$ is Lebesgue integrable.
	Therefore, we can see that $\widehat\chi(-\xi) \chi(\epsilon^2 \xi)$ is dominated by a Lebesgue integrable function.
	Thus by using Lebesgue Dominated Convergence Theorem, we can continue \eqref{eq:xxiInt1-SuppMater} as
	\begin{align*}
	\int_{\R_x^n \times \R_\xi^n} e^{ix\cdot \xi} \dif x \dif \xi
	& = (2\pi)^{n/2} \int \widehat\chi(-\xi) \dif \xi
	= (2\pi)^{n} \chi(0) = (2\pi)^{n}.
	\end{align*}
	We arrive at \eqref{eq:xxiInt-SuppMater}.

	To show, we first show that	\begin{equation} \label{eq:Lemma3-NashMoserStdyNt}
	(2\pi)^{-n} \iint e^{-iy\cdot \eta} y^\alpha \eta^\beta \dif y \dif \eta = (2\pi)^{-n} \iint e^{-iy\cdot \eta} D_\eta^\alpha(\eta^\beta) \dif y \dif \eta,
	\end{equation}
	where $D_{\eta_j} := \frac 1 {\textrm{i}} \partial_{\eta_j}$.
	Both the LHS and RHS in \eqref{eq:Lemma3-NashMoserStdyNt} should be understood as oscillatory integrals. Thus fix some $\chi \in \scrD(\Rn)$ such that $\chi(x) \equiv 1$ when $|x| \leq 1$, we have
	\begin{align}
	\iint e^{-iy\cdot \eta} y^\alpha \eta^\beta \dif y \dif \eta 
	& = \lim_{\epsilon \to 0^+} \iint e^{-iy\cdot \eta} y^\alpha \eta^\beta \chi(\epsilon y) \chi(\epsilon \eta) \dif y \dif \eta \nonumber \\
	& =\lim_{\epsilon \to 0^+} \sum_{0 < \gamma \leq \alpha} \epsilon^{|\gamma|} \binom{\alpha}{\gamma} \iint e^{-iy\cdot \eta} \chi(\epsilon y) \cdot D_\eta^{\alpha-\gamma}(\eta^\beta) \cdot \big( \partial^\gamma \chi \big) (\epsilon \eta) \dif y \dif \eta \nonumber \\
	& \quad + \iint e^{-iy\cdot \eta} D_\eta^\alpha \big( \eta^\beta \big) \dif y \dif \eta. \label{eq:Lemma3init-NashMoserStdyNt}
	\end{align}
	As $\epsilon$ goes to zero, we have
	\begin{align*}
	\iint e^{-iy\cdot \eta} \chi(\epsilon y) \cdot D_\eta^{\alpha-\gamma}(\eta^\beta) \cdot \big( \partial^\gamma \chi \big) (\epsilon \eta) \dif y \dif \eta 
	\to & \ D_\eta^{\alpha-\gamma}(\eta^\beta) \cdot \big( \partial^\gamma \chi \big) (\epsilon \eta) \big|_{\eta = 0} \quad (\epsilon \to 0^+).
	\end{align*}
	Because $\gamma > 0$, $\big( \partial^\gamma \chi \big) (\epsilon \eta) \big|_{\eta = 0} = 0$. Therefore, we have
	\begin{equation}
	\lim_{\epsilon \to 0^+} \epsilon^{|\gamma|} \sum_{0 < \gamma \leq \alpha} \binom{\alpha}{\gamma} \iint e^{-iy\cdot \eta} \chi(\epsilon y) \cdot D_\eta^{\alpha-\gamma}(\eta^\beta) \cdot \big( \partial^\gamma \chi \big) (\epsilon \eta) \dif y \dif \eta = 0. \label{eq:Lemma3Inter-NashMoserStdyNt}
	\end{equation}
	Combining \eqref{eq:Lemma3init-NashMoserStdyNt} and \eqref{eq:Lemma3Inter-NashMoserStdyNt}, we arrive at
	$$\iint e^{-iy\cdot \eta} y^\alpha \eta^\beta \dif y \dif \eta = \lim_{\epsilon \to 0^+} \iint e^{-iy\cdot \eta} y^\alpha \eta^\beta \chi(\epsilon y) \chi(\epsilon \eta) \dif y \dif \eta = \iint e^{-iy\cdot \eta} D_\eta^\alpha \big( \eta^\beta \big) \dif y \dif \eta.$$
	We proved \eqref{eq:Lemma3-NashMoserStdyNt}.

	Then, for multi-indexes $\alpha$ and $\beta$, if there exists $i$ such that $\alpha_i \neq \beta_i$, say, $\alpha_i > \beta_i$, then $D_\xi^\alpha(\xi^\beta) = 0$ and so
	\begin{equation*}
		\int e^{ix\cdot \xi} x^\alpha \xi^\beta \dif x \dif \xi 
		= \int e^{ix\cdot \xi} D_\xi^\alpha(\xi^\beta) \dif x \dif \xi
		= 0.
	\end{equation*}

	When $\alpha = \beta$, we have
	\begin{equation*}
	\int e^{ix\cdot \xi} x^\alpha \xi^\beta \dif x \dif \xi 
	= \int e^{ix\cdot \xi} D_\xi^\alpha(\xi^\alpha) \dif x \dif \xi
	= \int e^{ix\cdot \xi} \alpha! \dif x \dif \xi
	= (2\pi)^n \alpha!.
	\end{equation*}
	We have arrived at \eqref{eq:xxiabInt-SuppMater}.
\end{proof}

We also need \cite[Lemma 18.2.1]{hormander1985analysisIII} and we present a proof below.

\begin{lem} \label{lem:aat-RSRe2020}
If $a \in S^m(\Rn \times \R^k)$ and $u$ is defined by the oscillatory integral
	\[
	u(x) = \int e^{i \agl[x',\xi']} a(x,\xi') \dif \xi',
	\]
	then there exists $\tilde a \in S^m(\R^{n-k} \times \R^k)$ such that
	\[
	u(x) = \int e^{i \agl[x',\xi']} \tilde a(x'',\xi') \dif \xi',
	\]
	and $\tilde a$ has the asymptotic expansion
	\[
	\tilde a(x'',\xi') \sim \sum_{|\alpha| \leq N} \partial_{x'}^\alpha \partial_{\xi'}^\alpha a(0,x'',\xi') / \alpha!.
	\]
\end{lem}

\begin{rem} \label{rem:aat-RSRe2020}
Note that if $a(x,\xi') = 0$ near $\{x' = 0\}$, e.g.~$a(x,\xi) = (1 - \chi(x')) a'(x,\xi)$ for some $a'$ and some cutoff function satisfying $\chi(y) \equiv 1$ near the origin, 
then Lemma \ref{lem:aat-RSRe2020} implies that $\tilde a \in S^{-\infty}$.
\end{rem}

\begin{proof}
	The $\tilde a(x'',\cdot)$ is the Fourier transform of $u(\cdot,x'')$ with some constants, i.e.
	\[
	\tilde a(x'',\xi') = (2\pi)^{-k/2} \calF_{x'} \{ u(x',x'') \}(\xi') = (2\pi)^{-k} \int e^{-ix' \cdot \xi'} u(x',x'') \dif x'.
	\]
	Then we can have
	\begin{align*}
	\tilde a(x'',\xi') 
	& = (2\pi)^{-k} \int e^{-ix' \cdot \xi'} u(x) \dif x'
	= (2\pi)^{-k} \int e^{ix' \cdot \theta} a(x,\xi' + \theta) \dif \theta \dif x'.
	\end{align*}
	By adopting the way used in \S I.8.1 in \cite{alinhac2007pseudo} in computing the oscillatory integral, we can easily show that ${|\partial_{x''}^\alpha \partial_{\xi'}^\beta \tilde a(x'',\xi)| \lesssim \agl[\xi']^{m - |\beta|}}$, and this can be seen by the fact that
\[
|\partial_{\xi'}^\alpha [\chi(x,\theta) a(2^k x',x'',\xi' + 2^k \theta)]| \lesssim 2^{mk} \agl[\xi']^k,
\]
so $\tilde a \in S^m(\R^{n-k} \times \R^k)$.

The idea of the proof is to expand $a(x',x'', \xi' + \theta)$ in terms of $x'$ and $\theta$ by Taylor expansion
\begin{align} 
a(x',x'',\xi' + \theta) 
& = \sum_{|\alpha| + |\beta| \leq 2N} \frac {x'^\alpha  \theta^\beta}{\alpha! \beta!} \partial_{x'}^\alpha \partial_{\xi'}^\beta a(0, x'',\xi') \nonumber \\
& \quad + \sum_{|\alpha| + |\beta| = 2N+1} \frac {x'^\alpha \theta^\beta}{\alpha! \beta!} \partial_{x'}^\alpha \partial_{\xi'}^\beta a(\eta x', x'',\xi' + \eta \theta), \quad 0 < \eta < 1, \nonumber 
\end{align}
and to use Lemma \ref{lem:intab-RSRe2020}.
We have
\begin{align}
\tilde a(x'', \xi')
& = (2\pi)^{-k} \int e^{ix' \cdot \theta} a(x',x'',\xi' + \theta) \dif \theta \dif x' \nonumber \\
\nonumber \\
& = \sum_{|\alpha| \leq N} \partial_{x'}^\alpha \partial_{\xi'}^\beta a(0, x'',\xi') / \alpha! \nonumber \\
& \quad + \sum_{\substack{ |\alpha| + |\beta| = 2N+1 \\ \gamma \leq \alpha \leq \beta + \gamma}} C_{\alpha,\beta,\gamma} \int e^{ix' \cdot \theta} (\partial_{x'}^\alpha \partial_{\xi'}^{\beta + \gamma} a) (\eta^{|\gamma|}) \partial_\theta^{\alpha - \gamma} (\theta^\beta) \dif \theta \dif x' \label{eq:taa4-StudyNoteHormander}.
\end{align}
Note that the constraint $\alpha \leq \beta + \gamma$ in \eqref{eq:taa4-StudyNoteHormander} comes from the fact that $\partial_\theta^{\alpha - \gamma} (\theta^\beta) = 0$ when $\alpha > \beta + \gamma$.
Moreover, the constraint ``$|\alpha| + |\beta| = 2N+1,\, \gamma \leq \alpha \leq \beta + \gamma$'' gives
\begin{equation*} 
2N + 1 = |\alpha| + |\beta| \leq 2|\beta| + |\gamma| \leq 2(|\beta| + |\gamma|)
\quad \Rightarrow 
\quad 
|\beta + \gamma| \geq N + 1.
\end{equation*}

Now we show that each remainder term in \eqref{eq:taa4-StudyNoteHormander} is controlled by $\agl[\xi']^{m-N-1}$.
Denote $b(x', x'', \theta; \xi',\eta) = \big( \partial_{x'}^\alpha \partial_{\xi'}^{\beta + \gamma} a(\eta x', x'', \xi' + \eta \theta) \big) (\eta^{|\gamma|}) \partial_\theta^{\alpha - \gamma} (\theta^\beta)$ with underlining assumptions $\beta + \gamma \geq \alpha$ and $|\beta + \gamma| \geq N + 1$, and we have
\begin{align*}
\tilde a(x'', \xi') - \sum_{|\alpha| \leq N} \partial_{x'}^\alpha \partial_{\xi'}^\beta a(0, x'',\xi') / \alpha!
& = \int e^{ix' \cdot \theta} \chi_0(x',\theta) b \dif \theta \dif x' \nonumber \\
& \quad + \sum_{\ell \geq 1} \int e^{ix' \cdot \theta} \chi(x'/2^\ell,\theta/2^\ell) b(x', x'', \theta; \xi',\eta) \dif \theta \dif x',
\end{align*}
where $\chi_0$ and $\chi$ is as in \cite[\S I.8.1]{alinhac2007pseudo}.
Here we only show how the second term in the equation above is controlled by $\agl[\xi']^{m-N-1}$.
The computation is as follows,
	\begin{align}
	& \int e^{ix' \cdot \theta} \chi(x'/2^\ell,\theta/2^\ell) b(x', x'', \theta; \xi',\eta) \dif \theta \dif x' \nonumber\\
	\lesssim & \ 2^{2\ell k} \int \big( \frac {(\theta,x') \cdot \nabla_{(x',\theta)}} {i2^{2\ell} (|x'|^2+|\theta|^2)} \big)^L (e^{i 2^{2\ell} x' \cdot \theta}) \cdot \chi(x',\theta) b(x', x'', \theta; \xi', 2^\ell \eta) \dif \theta \dif x' \nonumber\\
	\lesssim & \ \agl[\xi']^{m-N-1} \cdot 2^{\ell(2k + 1 - 2L)} \int_{\supp \chi} C_L 2^{\ell (|m-N-1|+L)} \dif \theta \dif x' \nonumber\\
	\lesssim & \ \agl[\xi']^{m-N-1} \cdot 2^{\ell(2k + 1 + |m-N-1| - L)}, \nonumber
	\end{align}
	thus if we take $L$ to be large enough such that $2k + 1 + |m-N-1| - L < 0$, we can have
	\[
	|\sum_{\ell \geq 1} \int e^{ix' \cdot \theta} \chi(x'/2^\ell,\theta/2^\ell) b \dif \theta \dif x'| \lesssim \sum_{\ell \geq 1} \agl[\xi']^{m-N-1} 2^{\ell(2k + 1 + |m-N-1| - L)} \lesssim \agl[\xi']^{m-N-1}.
	\]
	This shows $|\tilde a(x'', \xi') - \sum_{|\alpha| \leq N} \partial_{x'}^\alpha \partial_{\xi'}^\alpha a(0,x'',\xi') / \alpha!| \lesssim \agl[\xi']^{m-N-1}$. Using the same procedure, we can easily show $\big| \partial_{x''}^\kappa \partial_{\xi'}^\beta [\tilde a(x'', \xi') - \sum_{|\alpha| \leq N} \partial_{x'}^\alpha \partial_{\xi'}^\alpha a(0,x'',\xi') / \alpha!] \big| \lesssim \agl[\xi']^{m-N-1 - |\beta|}$, and hence
\[
	\tilde a(x'', \xi') - \sum_{|\alpha| \leq N} \partial_{x'}^\alpha \partial_{\xi'}^\alpha a(0,x'',\xi') / \alpha! \in S^{m-N-1}(\R^{n-k} \times \R^k).
\]
The proof is complete.
\end{proof}

We also need \cite[Lemma 18.2.9]{hormander1985analysisIII} and we present a proof below.

\begin{lem} \label{lem:kka-RSRe2020}
Assume that $a \in S^m$ and
\[
u(x) = \int e^{i \agl[x', \xi']} a(x,\xi') \dif \xi', \quad \xi' \in \R^k,
\]
and a $C^\infty$ diffeomorphism $\rho \colon y \in \Rn \mapsto \rho(y) = (\rho_1(y), \rho_2(y)) \in \Rn$ preserving the hyperplane $S = \{ x \,;\, x' = 0 \}$.
The $\rho_1$ is $k$-dimensional while $\rho_2$ is $(n-k)$-dimensional.
Assume $u$ and the pull-back $\rho^* u$ is $C^\infty$-smooth in $\Rn \backslash S$,
then there exists $\tilde a \in S^m(\R^{n-k} \times \R^k)$ such that $\rho^* u$ can be represented as
\[
\rho^* u(y) = \int e^{i \agl[y', \xi']} \tilde a(y'',\xi') \dif \xi',
\]
and
\[
\tilde a(y'',\eta) - a(0,\rho_2(0,y''),(\psi(0,y''))^{T,-1} \eta) |\det\psi(0,y'')|^{-1} \in S^{m-1}(\R^{n-k} \times \R^k),
\]
where $(*)^T$ and $(*)^{T,-1}$ signify the transpose and transpose with inverse of a matrix, respectively.
\end{lem}

\begin{rem}
The condition ``$u$ and $\rho^* u$ is $C^\infty$-smooth in $\Rn \backslash S$'' is indispensable. 
\end{rem}

\begin{proof}
Because $\rho$ preserves the hyperplane $\{ x \,;\, x' = 0 \}$, there exists a $C^\infty$ matrix-valued function $\psi$ such that $\rho_1(y',y'') = \psi(y) \cdot y'$, where the dot operation ``$\cdot$'' here signifies the matrix multiplication.
According to Lemma \ref{lem:aat-RSRe2020}, there exist $\bar a \in S^m$ such that
\(
u(x) = \int e^{i \agl[x', \xi']} \bar a(x'',\xi') \dif \xi'.
\)
Hence we have
\begin{align*}
\rho^* u(y)
& = u(\rho(y))
= \int e^{i \agl[\rho_1(y), \xi']} \bar a(\rho_2(y),\xi') \dif \xi'
= \int e^{i \agl[\psi(y) \cdot y', \xi']} \bar a(\rho_2(y),\xi') \dif \xi' \nonumber \\
& = \int e^{i \agl[y', (\psi(y) )^{T} \xi']} \bar a(\rho_2(y),\xi') \dif \xi', \nonumber 
\end{align*}
According to Remark \ref{rem:aat-RSRe2020}.
Therefore, we could continue
\begin{align*}
\tilde u(y)
& = \int e^{i \agl[y', (\psi(y))^{T} \xi']} \chi(y') \bar a(\rho_2(y),(\psi(y))^{T,-1} (\psi(y))^{T} \xi') |\det\psi(y)|^{-1} \dif ((\psi(y) )^{T} \xi') + v(y) \nonumber \\
& = \int e^{i \agl[y', \eta]} \chi(y') \bar a(\rho_2(y),(\psi(y))^{T,-1} \eta) |\det\psi(y)|^{-1} \dif \eta + v(y), \nonumber
\end{align*}
where $\chi \in C_c^\infty(\R^k)$ with $\chi(y') \equiv 1$ in a neighborhood $0$ such that the matrix $\psi(y)$ is invertible in $\supp \psi$, and $v(y) = \int e^{i \agl[y', \eta]} b(y'',\eta) \dif \eta$ with $b \in S^{-\infty}$.
Using Lemma \ref{rem:aat-RSRe2020}, we obtain
\(
\tilde u(y) = \int e^{i \agl[y', \eta]} \tilde a(y'', \eta) \dif \eta
\)
where
\[
\tilde a(y'',\eta) - \bar a(\rho_2(0,y''),(\psi(0,y''))^{T,-1} \eta) |\det\psi(0,y'')|^{-1} \in S^{m-1}(\R^{n-k} \times \R^k).
\]
Note that $\bar a$ satisfies
\(
\bar a(x'',\xi') - a(0,x'',\xi') \in S^{m-1}(\R^{n-k} \times \R^k),
\)
so
\[
\tilde a(y'',\eta) - a(0,\rho_2(0,y''),(\psi(0,y''))^{T,-1} \eta) |\det\psi(0,y'')|^{-1} \in S^{m-1}(\R^{n-k} \times \R^k).
\]
The proof is complete.
\end{proof}

Finally, we need Lemma \ref{lem:intConv-RSRe2020}.

\begin{lem} \label{lem:intConv-RSRe2020}
	For any stochastic process $\{g(k,\omega)\}_{k \in \R_+}$ satisfying
\begin{equation} \label{eq:E1-RSRe2020}
	\int_1^{+\infty} k^{m-1} \mathbb E (|g(k,\cdot)|) \dif k < +\infty,
\end{equation}
it holds that
\begin{equation} \label{eq:E2-RSRe2020}
	\lim_{K \to +\infty} \frac 1 K \int_K^{2K} k^m g(k,\omega) \dif k = 0, \ \ass \omega \in \Omega.
\end{equation}
\end{lem}

\begin{proof}
	Check \cite[Lemma 4.1]{ma2019determining}.
\end{proof}

\subsection{Key steps in the proof} \label{sec:MKS-RSRe2020}

Lemma \ref{lem:intConv-RSRe2020} turns the justification of the ergodicity into the asymptotic analysis of the expectation of related terms.

With the help of Lemma \ref{lem:intConv-RSRe2020}, the most difficult part of the work \cite{LassasA, Lassas2008, LiHelinLiinverse2018, li2019inverse, li2020inverse} boils down to the estimate of the integral
\begin{equation} \label{eq:Ixyz-RSRe2020}
\mathbb I (x,y,k_1,k_2) := \int e^{ik_1(|x - z_1| + |z_1 - y|) - ik_2(|x - z_2| + |z_2 - y|)} C(z_1,z_2) \dif z_1 \dif z_2,
\end{equation}
where $C(z_1,z_2) = \int e^{i (z_1 - z_2) \cdot \xi} c(z_1,\xi) \dif \xi$ and $c \in S^{-m}$.
Readers may refer to
\cite[(30)-(31)]{Lassas2008},
\cite[(3.21) and (3.24)]{LiHelinLiinverse2018}, \cite[(4.2) and (2.1)]{li2019inverse} as well as \cite[Theorems 3.1 and 3.3]{li2020inverse} as examples.

One wonders the decaying rate of $\mathbb I$ in terms of $k_1$ and $k_2$,
and after we got the decaying rate, we substitute this estimate into \eqref{eq:E1-RSRe2020}.
If $\mathbb I$ decays fast enough in terms of $k_1$ and/or $k_2$, the corresponding integral in \eqref{eq:E1-RSRe2020} will be finite and we can obtain some asymptotic ergodicity like \eqref{eq:E2-RSRe2020}.
This is the principal idea in \cite{LassasA, Lassas2008, LiHelinLiinverse2018, li2019inverse, li2020inverse}.

\begin{prop} \label{prop:Ixy-RSRe2020}
	Assume $\mathbb I$ is defined as in \eqref{eq:Ixyz-RSRe2020} and $C(z_1,z_2) = \int e^{i (z_1 - z_2) \cdot \xi} c(z_1,\xi) \dif \xi$ with $c \in S^{-m}$ is a symbol.
	Then for $\forall N \in \mathbb N$ there exists constants $C_N > 0$ such that
	\[
	|\mathbb I (x,y,k_1,k_2)| \leq C_N \agl[k_1 - k_2]^{-N} (k_1 + k_2)^{-m},
	\]
	holds uniformly for $x$, $y$.
\end{prop}

\begin{proof}
Denote $\phi(z_1,z_2, x,y, k_1, k_2) := k_1(|x - z_1| + |z_1 - y|) - k_2(|x - z_2| + |z_2 - y|)$, then $\mathbb I = \int e^{i\phi} C \dif z_1 \dif z_2$ and $\phi$ is the phase function.
We have
\begin{align} 
\phi(z_1,z_2, x,y, k_1, k_2) 
& = \frac {k_1 + k_2} 2 \big[ (|x - z_1| + |z_1 - y|) - (|x - z_2| + |z_2 - y|) \big] \nonumber \\
& \quad + \frac {k_1 - k_2} 2 \big[ (|x - z_1| + |z_1 - y|) + (|x - z_2| + |z_2 - y|) \big]. \label{eq:kk12-RSRe2020}
\end{align}
We note that the $xyz$ part of the  second term in \eqref{eq:kk12-RSRe2020} is always positive and the first term equals to zero when $z_1 = z_2$.
Also, the function $C$ will be singular when $z_1 = z_2$.
Therefore, the situation near the hyperplane $S_0 := \{ z_1 = z_2 \}$ is crucial for the behavior of $\mathbb I$ regarding the decaying rate in terms of $k_1$, $k_2$.
Therefore, we are willing to do a change of variables inside the integral \eqref{eq:Ixyz-RSRe2020} such that the hyperplane $S_0$ can be featured by a single variable, i.e.~$S_0 = \{v = 0\}$ for some variable $v$.
To be specific,
we choose the change of variables $\tau_1(z_1, z_2) = (v,w)$ where
\begin{equation*} 
\tau_1 \colon \quad v = z_1 - z_2, \quad w = z_1 + z_2.
\end{equation*}
The pull-back of $C$ under $\tau_1^{-1}$ is
\begin{equation} \label{eq:C1-RSRe2020}
C_1(v,w) := (\tau_1^{-1})^* C(v,w) = C(\tau_1^{-1}(v,w))
= \int e^{i v \cdot \xi} c((v + w)/2,\xi) \dif \xi.
\end{equation}

Second, in order to make the phase function $\phi$ more easy to handle, we are also willing to do another change of variables such that $\phi$ can be represented in the form of inner products, i.e.~$\phi = s \cdot t$ for some $s$ and $t$ depending on $x$, $y$, $z_1$, $z_2$, $k_1$ and $k_2$.
One of the choices is $\tau_2(z_1, z_2) = (s,t)$, $s = (s_1,\cdots, s_n) \in \Rn$ and $t = (t_1,\cdots, t_n) \in \Rn$ where
\begin{equation} \label{eq:tau2-RSRe2020}
\tau_2 \colon \quad 
\left\{ \begin{aligned}
s_1 & = (|x - z_1| + |z_1 - y|) - (|x - z_2| + |z_2 - y|), \\
t_1 & = (|x - z_1| + |z_1 - y|) + (|x - z_2| + |z_2 - y|).
\end{aligned} \right.
\end{equation}
We comment that under \eqref{eq:tau2-RSRe2020}, the phase function $\phi$ will only depend on $s_1$ and $t_1$,
and the choice of $s_j$ and $t_j~(j = 2,\cdots, n)$ is inessential as long as the change of variables $\tau_2$ is a diffeomorphism.
Hence we omit the precise definitions of  $s_j$ and $t_j$ $(j > 1)$ and readers may refer to \cite{LassasA, Lassas2008, LiHelinLiinverse2018, li2019inverse, li2020inverse} for more details.
Another thing to note is the map $\tau_1 \circ \tau_2^{-1}$ preserves $S_0$, i.e.~$\tau_1 \circ \tau_2^{-1}(0,t) = (0,w)$.
By Lemma \ref{lem:kka-RSRe2020}, there exists a symbol $c_2 \in S^{-m}$ such that the pull-back of $C_1$ under $\tau_1 \circ \tau_2^{-1}$ is
\begin{equation} \label{eq:C2-RSRe2020}
C_2(s,t) := (\tau_1 \circ \tau_2^{-1})^* C_1(s,t)
= \int e^{i s \cdot \xi} c_2(t,\xi) \dif \xi,
\end{equation}
By using Lemma \ref{lem:kka-RSRe2020},
we can express $c_2$ by $c$, $\tau_1$ and $\tau_2$, which involves some detailed computations.
Note that we only need the leading term of $c_2$ so the computations wouldn't be too complicated.

The relationship \eqref{eq:C2-RSRe2020} also gives
\[
C_2(s,t)
= (\tau_1 \circ \tau_2^{-1})^* (\tau_1^{-1})^* C(s,t)
= (\tau_2^{-1})^* C(s,t),
\]
and hence we can do the change of variables $\tau_2$ in \eqref{eq:Ixyz-RSRe2020} to obtain
\begin{align}
\mathbb I (x,y,k_1,k_2)
& = \int e^{ik_1(|x - z_1| + |z_1 - y|) - ik_2(|x - z_2| + |z_2 - y|)} C(\tau_2^{-1} \circ \tau_2(z_1,z_2)) \dif (\tau_2^{-1} \circ \tau_2(z_1,z_2)) \nonumber \\
& = \int e^{i(k_1 + k_2) s_1 / 2 + i(k_1 - k_2) t_1 / 2} C(\tau_2^{-1} (s,t)) |\det \tau_2^{-1}(s,t)| \dif (s,t) \nonumber \\
& = \int e^{i(k_1 + k_2) s \cdot e_1 / 2 + i(k_1 - k_2) t \cdot e_1 / 2} C_2(s,t) |\det \tau_2^{-1}(s,t)| \dif s \dif t. \label{eq:Ixyz2-RSRe2020}
\end{align}
Here we need the help of Lemma \ref{lem:aat-RSRe2020} to deal with the $|\det \tau_2^{-1}(s,t)|$ term: there exists a symbol $\tilde c_2 \in S_{-m}$ such that
\begin{equation} \label{eq:tc2-RSRe2020}
C_2(s,t) |\det \tau_2^{-1}(s,t)| = \int e^{i s \cdot \xi} \tilde c_2(t,\xi) \dif \xi.
\end{equation}
The computation of the leading term of $\tilde c_2$ is straight forward,
\[
\tilde c_2(t,\xi)
- c_2(t,\xi) |\det \tau_2^{-1}(0,t)| \in S^{-m-1}.
\]
Combining \eqref{eq:Ixyz2-RSRe2020} and \eqref{eq:tc2-RSRe2020}, we arrive at
\begin{align*}
\mathbb I (x,y,k_1,k_2)
& = \int e^{i(k_1 + k_2) s \cdot e_1 / 2 + i(k_1 - k_2) t \cdot e_1 / 2} \int e^{i s \cdot \xi} \tilde c_2(t,\xi) \dif \xi \dif s \dif t \\
& \simeq \int e^{i(k_1 - k_2) t \cdot e_1 / 2} \tilde c_2(t,-(k_1 + k_2)e_1 / 2) \dif t.
\end{align*}
Now we can see $\mathbb I$ is decaying at the rate of $\agl[k_1 - k_2]^{-N} (k_1 + k_2)^{-m}$ for arbitrary $N \in \mathbb N$.
\end{proof}

We would like to comment that the estimation of $\mathbb I$ is difficult due to the presence of the norm inside the phase function $\phi$.
However, the designs of $\tau_1$ and $\tau_2$ in the arguments above are so peculiar that the estimate of $\mathbb I$ is possible.

\section{Recovery by far-field data} \label{sec:Ma-RSRe2020}

In this section we consider key steps in the works
\cite{ma2020determining, ma2019determining}.
In \cite{ma2020determining, ma2019determining}, the authors use far-field data to do the recovery, and this makes the derivations different from what has been discussed in Section \ref{sec:Matti-RSRe2020}.
Different methodology is required to obtain accurate estimate of the decaying rate.
Lemma \ref{lem:FracExp-MLmGWsSchroEqu2019} plays a key role in the derivation.
Before stepping into the key steps in the derivation, we shall first investigate some useful lemmas.

\subsection{Useful lemmas} \label{subsec:usfl-RSRe2020}

 First, let us recall the notion of the fractional Laplacian \cite{pozrikidis2016fractional} of order 
$s \in (0,1)$ in $\mathbb{R}^n$ ($n\geq 3$),
\begin{equation} \label{eq:FracLapDef-MLmGWsSchroEqu2019}
(-\Delta)^{s/2} \varphi(x) := (2\pi)^{-n} \iint e^{{\textrm{i}}(x-y) \cdot \xi} |\xi|^s \varphi(y) \dif y \dif \xi,
\end{equation}
where the integration is defined as an oscillatory integral.
When $\varphi \in \scrS(\Rn)$, \eqref{eq:FracLapDef-MLmGWsSchroEqu2019} can be understood as a usual Lebesgue integral if one integrates w.r.t.~$y$ first and then integrates w.r.t.~$\xi$.
By duality arguments, the fractional Laplacian can be generalized to act on wider range of functions and distributions (cf. \cite{wong2014pdo}).
It can be verified that the fractional Laplacian is self-adjoint. 

In the following two lemmas, we present the results in a more general form where the space dimension $n$ can be arbitrary but greater than 2, though only the case $n=3$ shall be used subsequently. 

\begin{lem} \label{lem:FracExp-MLmGWsSchroEqu2019}
	For any $s \in (0,1)$, we have
	\begin{equation*}
	(-\Delta_\xi)^{s/2} (e^{{\textrm{i}}x \cdot \xi}) = |x|^s e^{{\textrm{i}}x \cdot \xi}
	\end{equation*}
	in the distributional sense.
\end{lem}

\begin{proof}
	Check \cite[Lemma 3.1]{ma2019determining}. 
\end{proof}


\begin{lem} \label{cor:FracSymbol-MLmGWsSchroEqu2019}
	For any $m < 0$ and $s \in (0,1)$, we have
	\[
	\big( (-\Delta_\xi)^{s/2} c \big) (x,\xi) \in S^{m-s}
	\quad \text{for any} \quad
	c(x,\xi) \in S^m.
	\]
\end{lem}

\begin{proof}
	Check \cite[Corollary 3.1]{ma2019determining}. 
\end{proof}

In the sequel, we denote $\diam(\Omega) := \sup\limits_{x, x' \in \Omega} \{ |x - x'|\}$.

\begin{lem} \label{lem:intbdd-MLSchroEqu2018}
Assume $\Omega$ is a bounded domain in $\Rn$.
For $\forall \alpha, \beta \in \R$ such that $\alpha < n$ and $\beta < n$, and for $\forall p \in \Rn \backslash \{0\}$, there exists a constant $C_{\alpha,\beta}$ independent of $p$ and $\Omega$ such that
\begin{equation*} 
\int_\Omega |t|^{-\alpha} |t - p|^{-\beta} \dif t 
\leq C_{\alpha,\beta} \times
\begin{cases} |p|^{n - \alpha - \beta} + (\diam (\Omega))^{n - \alpha - \beta}, & \alpha + \beta \neq n, \\
\ln \frac 1 {|p|} + \ln (\diam(\Omega)) +C_{\alpha, \beta}, & \alpha + \beta = n.
\end{cases}
\end{equation*}
\end{lem}

\begin{proof}
Check \cite[Lemma 3.5]{ma2020determining}.
\end{proof}

\subsection{Key steps in \cite{ma2019determining}} \label{subsec:MKS2-RSRe2020}

In this subsection we restrict ourselves to $\R^3$.
One of the key difficulty in \cite{ma2019determining} is to obtain an asymptotics about a integral
\begin{equation} \label{eq:bbJ-RSRe2020}
\mathbb J
:= \int e^{\textrm{i}k\varphi(y,s,z,t)} \big( \int e^{\textrm{i}(z-y) \cdot \xi} c_q(z,\xi) \dif \xi \big) \big( \int e^{\textrm{i}(t-s) \cdot \eta} c_f(t,\xi) \dif \eta \big) \cdot \dif{(s, y, t, z)},
\end{equation}
in terms of $k$, where $\varphi(y,s,z,t) := -\hat x \cdot (y-z) - |y-s| + |z-t|$, 
$c_q \in S^{-m_q}$ and $c_f \in S^{-m_f}$ with $m_q$, $m_f$ satisfying the requirement in Theorem \ref{thm:6-RSRe2020},
$\dif{(s, y, t, z)}$ is a short notation for $\dif s \dif y \dif t \dif z$, 
and $y,z \in D_q$ and $s,t \in D_f$ two convex domains $D_q$ and $D_f$ satisfying \eqref{eq:fqSeparation-MLmGWsSchroEqu2019}.
Recall the definition of the unit normal vector $\boldsymbol{n}$ after \eqref{eq:fqSeparation-MLmGWsSchroEqu2019}.
We introduce two differential operators with $C^\infty$-smooth coefficients as follows,
	\begin{equation*}
	L_1 := \frac {(y-s) \cdot \nabla_s} {\textrm{i}k|y-s|}, 
	\quad
	L_2 = L_{2,\hat x} := \frac {\nabla_y \varphi \cdot \nabla_y} {\textrm{i}k|\nabla_y \varphi|},
	\end{equation*}
	where ${\nabla_y \varphi = \frac {s-y} {|s-y|} - \hat x}$. The operator $L_{2,\hat x}$ depends on $\hat x$ because $\nabla_y \varphi$ does. 
	Due to the fact that $y \in D_q$ while $s \in D_f$, the operator $L_1$ is well-defined.
	It can be verified there is a positive lower bound of $|\nabla_y \varphi|$ for all $\hat x\in \{\hat x \in \mathbb{S}^2 \colon \hat x \cdot \boldsymbol{n} \geq 0\}$. 
	It can also be verified that
	\[
	L_1 (e^{\textrm{i}k\varphi(y,s,z,t)}) = L_2 (e^{\textrm{i}k\varphi(y,s,z,t)}) = e^{\textrm{i}k\varphi(y,s,z,t)}.
	\]

In what follows, we shall use ${\calC(\cdot)}$ and its variants, such as ${\vec \calC(\cdot)}$, ${\calC_{a,b}(\cdot)}$ etc., to represent some generic smooth scalar/vector functions, within $C_c^\infty(\R^3)$ or $C_c^\infty(\R^{3 \times 4})$, whose particular definition may change line by line.
By using integration by parts, one can compute
	\begin{align}
	\mathbb J = & \int \big( L_1^2 L_2^2 \big) (e^{\textrm{i}k\varphi(y,s,z,t)})
	\cdot 
	\big( \int e^{\textrm{i}(z-y) \cdot \xi} c_q(z,\xi) \dif \xi \big) \cdot 
	\big( \int e^{\textrm{i}(t-s) \cdot \eta} c_f(t,\eta) \dif \eta \big) \dif{(s, y, t, z)} \nonumber\\
	\simeq & \ k^{-4} \int_{\mathcal D} e^{\textrm{i}k\varphi(y,s,z,t)} \big[
	\mathcal J_1\, (\mathcal K_1\, \calC + \vec{\mathcal K}_2 \cdot \vec{\calC} + \sum_{a,b=1,2,3} \mathcal K_{3;a,b}\, \calC_{a,b}) \nonumber\\
	& \ + \sum_{c=1,2,3} \mathcal J_{2;c}\, (\mathcal K_1\, \calC_c + \vec{\mathcal K}_2 \cdot \vec{\calC}_c + \sum_{a,b=1,2,3} \mathcal K_{3;a,b}\, \calC_{a,b,c}) \nonumber\\
	& \ + \sum_{a',b'=1,2,3} \mathcal J_{3;a',b'} (\mathcal K_1\, \calC_{a',b'} + \vec{\mathcal K}_2 \cdot \vec{\calC}_{a',b'} + \sum_{a,b=1,2,3} \mathcal K_{3;a,b}\, \calC_{a,b,a',b'}) \big] \dif{(s, y, t, z)}, \label{eq:hotF1F1Inter2-MLmGWsSchroEqu2019}
	\end{align}
	where the integral domain $\mathcal D \subset \R^{3 \times 4}$ is bounded and
	\begin{alignat*}{2}
	& \mathcal J_1
	:= \int e^{\textrm{i}(t-s) \cdot \eta}\, c_f(t,\eta) \dif \eta,
	& \quad
	& \mathcal K_1
	:= \int e^{\textrm{i}(z-y) \cdot \xi}\, c_q(z,\xi) \dif \xi, \\
	& \vec{\mathcal J}_2
	:= \nabla_s \int e^{\textrm{i}(t-s) \cdot \eta}\, c_f(t,\eta) \dif \eta,
	& \quad 
	& \vec{\mathcal K}_2
	:= \nabla_y \int e^{\textrm{i}(z-y) \cdot \xi}\, c_q(z,\xi) \dif \xi, \\
	& \mathcal J_{3;a,b}
	:= \partial_{s_a,s_b}^2 \int e^{\textrm{i}(t-s) \cdot \eta}\, c_f(t,\eta) \dif \eta,
	& \quad
	& \mathcal K_{3;a,b}
	:= \partial_{y_a,y_b}^2 \int e^{\textrm{i}(z-y) \cdot \xi}\, c_q(z,\xi) \dif \xi,
	\end{alignat*}
	and $\mathcal J_{2;c}$ (resp.~$\mathcal K_{2;c}$) is the $c$-th component of the vector $\vec{\mathcal J}_2$ (resp.~$\vec{\mathcal K}_2$).

Here we only show how to estimate $\mathcal J_1$ and skip the details regarding $\vec{\mathcal J}_2$, $\mathcal K_1$, and $\vec{\mathcal K}_2$;
readers may refer to the proof of \cite[Lemma 3.3]{ma2019determining} for details.
For the case where $s \neq t$, we have
	\begin{align}
	|\mathcal J_1|
	& = |\int e^{\textrm{i}(t-s) \cdot \eta}\, c_f(t,\eta) \dif \eta|
	= |s-t|^{-2} \cdot |\int \Delta_\eta (e^{\textrm{i}(s-t) \cdot \eta})\, c_f(t,\eta) \dif \eta| \nonumber\\
	& = |s-t|^{-2} \cdot |\int e^{\textrm{i}(t-s) \cdot \eta} (\Delta_\eta c_f) (t,\eta) \dif \eta|
	\leq |s-t|^{-2} \int |(\Delta_\eta c_f) (t,\eta)| \dif \eta \nonumber\\
	& \lesssim |s-t|^{-2} \int \agl[\eta]^{-m_f-2} \dif \eta
	\lesssim |s-t|^{-2}. \label{eq:hotF1F1J1-MLmGWsSchroEqu2019}
	\end{align}
Similarly, we can have
\begin{equation} \label{eq:JK12-RSRe2020}
|\mathcal J_1|,\, |\vec{\mathcal J}_2|,\, |\mathcal K_1|,\, |\vec{\mathcal K}_2| \lesssim |y-z|^{-2}.
\end{equation}

But for $J_{3;a,b}$, if we mimic the derivation \eqref{eq:hotF1F1J1-MLmGWsSchroEqu2019}, then
	\begin{align}
	\mathcal J_{3;a,b}
	& 
	\simeq \int e^{\textrm{i}(t-s) \cdot \eta} \cdot c_f(t,\eta) \eta_a \eta_b \dif \eta
	\simeq |s-t|^{-2} \int \Delta_\eta (e^{\textrm{i}(t-s) \cdot \eta}) \cdot c_f(t,\eta) \eta_a \eta_b \dif \eta \nonumber\\
	& = |s-t|^{-2} \int e^{\textrm{i}(t-s) \cdot \eta} \cdot \Delta_\eta (c_f(t,\eta) \eta_a \eta_b) \dif \eta. \label{eq:hotF1F1J3Inter1-MLmGWsSchroEqu2019}
\end{align}
Note that $\Delta_\eta (c_f(t,\eta) \eta_a \eta_b) \in S^{-m_f}$ and thus is not absolutely integrable in $\R^3$.
If we further differentiate the term $e^{\textrm{i}(t-s) \cdot \eta}$ in \eqref{eq:hotF1F1J3Inter1-MLmGWsSchroEqu2019} by $\frac {\textrm{i}(s-t) \cdot} {|s-t|^2} \nabla_\eta$ and then {transfer} the operator $\nabla_\eta$ onto $\Delta_\eta (c_f(t,\eta) \eta_a \eta_b)$ by using integration by parts, we would arrive at
	\begin{equation*}
	|\mathcal J_{3;a,b}| \lesssim |s-t|^{-3} \int |\nabla_\eta \Delta_\eta (c_f(t,\eta) \eta_a \eta_b)| \dif \eta \leq |s-t|^{-3} \int \agl[\eta]^{-m_f-1} \dif \eta.
	\end{equation*}
	The term $\int \agl[\eta]^{-m_f-1} \dif \eta$ is absolutely integrable now, but the term $|s-t|^{-3}$ is not integrable at the hyperplane $s = t$ in $\R^3$.
	{To circumvent this dilemma, the fractional Laplacian can be applied as follows.}
	By using Lemma \ref{lem:FracExp-MLmGWsSchroEqu2019} and \ref{cor:FracSymbol-MLmGWsSchroEqu2019}, we can continue \eqref{eq:hotF1F1J3Inter1-MLmGWsSchroEqu2019} as
	\begin{align}
	|\mathcal J_{3;a,b}|
	& \simeq |s-t|^{-2} \cdot \big| |s-t|^{-s} \int (-\Delta_\eta)^{s/2} (e^{\textrm{i}(t-s) \cdot \eta}) \cdot \Delta_\eta (c_f(t,\eta) \eta_j \eta_\ell) \dif \eta \big| \nonumber\\
	& = |s-t|^{-2-s} \cdot |\int e^{\textrm{i}(t-s) \cdot \eta} \cdot (-\Delta_\eta)^{s/2} \big( \Delta_\eta (c_f(t,\eta) \eta_j \eta_\ell) \big) \dif \eta| \nonumber\\
	& \lesssim |s-t|^{-2-s} \int \agl[\eta]^{-m_f+2-2-s} \dif \eta = |s-t|^{-2-s} \int \agl[\eta]^{-m_f-s} \dif \eta, \label{eq:hotF1F1J3Inter2-MLmGWsSchroEqu2019}
	\end{align}
	where the number $s$ is chosen to satisfy $\max\{0, 3 - m_f\} < s < 1$, and the existence of such a number $s$ is guaranteed by noting that $m_f > 2$. Therefore, we have
	\begin{subequations}
		\begin{numcases}{}
		-m_f - s < -3, \label{eq:hotF1F1J3s1-MLmGWsSchroEqu2019}\\
		-2 - s > -3. \label{eq:hotF1F1J3s2-MLmGWsSchroEqu2019}
		\end{numcases}
	\end{subequations}
	Thanks to the condition \eqref{eq:hotF1F1J3s1-MLmGWsSchroEqu2019}, we can continue \eqref{eq:hotF1F1J3Inter2-MLmGWsSchroEqu2019} as
	\begin{align}
	|\mathcal J_{3;a,b} |
	& \lesssim |s-t|^{-2-s} \int \agl[\eta]^{-m_f-s} \dif \eta
	\lesssim |s-t|^{-2-s}. \label{eq:hotF1F1J3-MLmGWsSchroEqu2019}
	\end{align}
	Using similar arguments, we can also conclude that $K_{3;a,b}| \lesssim |y-z|^{-2-s}$.

Combining \eqref{eq:hotF1F1Inter2-MLmGWsSchroEqu2019}, \eqref{eq:JK12-RSRe2020} and \eqref{eq:hotF1F1J3-MLmGWsSchroEqu2019}, we arrive at
	\begin{align}
	|\mathbb J| & \lesssim k^{-4}
	 \int_{\mathcal D} (|\mathcal J_1| + |\vec{\mathcal J}_2| + \sum_{a',b'=1,2,3} |\mathcal J_{3;a',b'}|) \cdot (|\mathcal K_1| + |\vec{\mathcal K}_2| + \sum_{a,b=1,2,3} |\mathcal K_{3;a,b}|) \dif{(s, y, t, z)} \nonumber\\
	& \lesssim k^{-4} \int_{\widetilde{\mathcal D}} |s - t|^{-2-s} \dif s \dif t \cdot \int_{\widetilde{\mathcal D}} |y - z|^{-2-s} \dif y \dif z \label{eq:hotF1F1Inter3-RSRe2020}
	\end{align}
	for some sufficiently large but bounded domain $\widetilde{\mathcal D} \subset \R^{3 \times 2}$ satisfying $\mathcal D \subset \widetilde{\mathcal D} \times \widetilde{\mathcal D}$.
	Note that the integral \eqref{eq:hotF1F1Inter3-RSRe2020} should be understood as a singular integral because of the presence of the singularities occurring when $s = t$ and $y = z$.
	By \eqref{eq:hotF1F1Inter3-RSRe2020} and \eqref{eq:hotF1F1J3s2-MLmGWsSchroEqu2019}, we can finally conclude 
\(
|\mathbb J| \lesssim k^{-4},
\)
as $k$ be large enough.

\subsection{Key steps in \cite{ma2020determining}} \label{sec:MKS2-RSRe2020}

In this subsection  we restrict ourselves to $\R^3$.
We note that in \eqref{eq:bbJ-RSRe2020}, the domains $D_q$ and $D_f$ are assumed to be separated by two convex hulls.
This condition is relaxed in \cite{ma2020determining} and the corresponding details in the proof is also modified.
One of the key difficulty in \cite{ma2020determining} is to obtain an asymptotics about a integral
\begin{equation} \label{eq:DefI-MLSchroEqu2018}
\mathbb K(x,y) := \iint_{D_f \times D_f} K_f(s,t) \Phi(s-y; k_1) \overline \Phi(t-x; k_2) \dif{s} \dif{t},
\end{equation}
where $K_f$ is the kernel of the covariance operator of the migr field $f$ (cf.~\eqref{eq:CK-MLmGWsSchroEqu2019}), and $\Phi$ is defined in the beginning of Section \ref{subsec:preli-SchroEqu2018}.
From \eqref{eq:DefI-MLSchroEqu2018} we have
\begin{align}
\mathbb K(z,y) 
& \simeq \iint_{\widetilde{\mathcal D} \times \widetilde{\mathcal D}} e^{ik_1|s-y| -ik_2|t-z|} \big( |s-y|^{-1} |t-z|^{-1} \int e^{i(s-t) \cdot \xi} c(s,\xi) \dif{\xi} \big)  \dif{s} \dif{t}. \label{eq:I.inter1-MLSchroEqu2018} 
\end{align} 
Define two differential operators 
\[
L_1 := \frac {(s-y) \cdot \nabla_s} {ik_1|s-y|}
\quad \text{and} \quad
L_2 := \frac {(t-z) \cdot \nabla_t} {-ik_2|t-z|}.
\]
It can be  verified that
\[
L_1 L_2 (e^{ik_1|s-y|-ik_2|t-z|}) = e^{ik_1|s-y|-ik_2|t-z|}.
\]
Hence, {noting that the integrand is compactly supported in $\widetilde{\mathcal D} \times \widetilde{\mathcal D}$} and by using integration by part, we can continue \eqref{eq:I.inter1-MLSchroEqu2018} as
		\begin{align}
& \ |\mathbb K(z,y) | \nonumber\\
\simeq & \ |\iint_{\widetilde{\mathcal D} \times \widetilde{\mathcal D}} L_1 L_2 (e^{ik_1|s-y| -ik_2|t-z|}) \big( |s-y|^{-1} |t-z|^{-1} \int e^{i(s-t) \cdot \xi} c_1(s,t,z,y,\xi) \dif{\xi} \big)  \dif{s} \dif{t}| \nonumber\\
\lesssim & \ k_1^{-1} k_2^{-1} \iint_{\widetilde{\mathcal D} \times \widetilde{\mathcal D}} \big[ |s-y|^{-2} |t-z|^{-2} \mathcal J_0 + |s-y|^{-2} |t-z|^{-1} (\max_a \mathcal J_{1;a}) \nonumber \\
& \hspace*{2cm} + |s-y|^{-1} |t-z|^{-2} (\max_{a}\mathcal J_{1;a}) + |s-y|^{-1} |t-z|^{-1} (\max_{a,b} \mathcal J_{2;a,b}) \big] \dif s \dif t, \label{eq:I.inter2-MLSchroEqu2018}
\end{align}
where $a,b$ are indices running from 1 to 3, and
\begin{align*}
\mathcal J_0
& := |\int e^{i(s-t) \cdot \xi}\, c_1(s,t,z,y,\xi) \dif \xi|, \\
\mathcal J_{1;a}
& := |\int e^{i(s-t) \cdot \xi}\, \xi_a c_1(s,t,z,y,\xi) \dif \xi|, \\
\mathcal J_{2;a,b}
& := |\int e^{i(s-t) \cdot \xi}\, \xi_a \xi_b c_1(s,t,z,y,\xi) \dif \xi|.\end{align*}

Because of the condition $m > 2$ (cf.~Theorem \ref{thm:5-RSRe2020}), we can find a  number $\tau \in (0,1)$ satisfying the inequalities $3 - m < \tau < 1$.
Therefore, we have
\begin{subequations} \label{eq:hotF1F1J3T-MLSchroEqu2018}
\begin{numcases}{}
- m - \tau < -3, \label{eq:hotF1F1J3s1-MLSchroEqu2018} \\
- 2 - \tau > -3. \label{eq:hotF1F1J3s2-MLSchroEqu2018}
\end{numcases}
\end{subequations}
By using Lemmas \ref{lem:FracExp-MLmGWsSchroEqu2019} and \ref{cor:FracSymbol-MLmGWsSchroEqu2019},
these quantities $\mathcal J_0$, $\mathcal J_{1;a}$ and $\mathcal J_{2;a,b}$ can be estimated as follows:	
\begin{align}
\mathcal J_0
& = |s-t|^{-\tau} \cdot |\int (-\Delta_\xi)^{\tau/2} (e^{i(s-t) \cdot \xi}) c_1 (s,t,z,y,\xi) \dif \xi| \nonumber \\
& = |s-t|^{-\tau} \cdot |\int e^{i(s-t) \cdot \xi}\, (-\Delta_\xi)^{\tau/2} (c_1 (s,t,z,y,\xi)) \dif \xi| \nonumber \\
& \lesssim |s-t|^{-\tau} \cdot \int \agl[\xi]^{-m-\tau} \dif \xi  \lesssim |s-t|^{-\tau}. \label{eq:hotF1F1J1-MLSchroEqu2018}
\end{align}
The last inequality in \eqref{eq:hotF1F1J1-MLSchroEqu2018} makes use of the fact \eqref{eq:hotF1F1J3s1-MLSchroEqu2018}.
Similarly, by first using fractional Laplacian and then using first-order differential operator on $e^{i(s-t) \cdot \xi}$, we can have
\begin{align}
\mathcal J_{1;a}
& \leq C |s-t|^{-1-\tau} \int \agl[\xi]^{-m+1-1-\tau} \dif \xi
\leq C |s-t|^{-1-\tau}, \label{eq:hotF1F1J2-MLSchroEqu2018} \\
\mathcal J_{2;a,b}
& \leq C |s - t|^{-2-\tau} |\int  \agl[\xi]^{-m+2-2-\tau} \dif \xi|
\leq C |s - t|^{-2-\tau}, \label{eq:hotF1F1J3-MLSchroEqu2018}
\end{align}
where the constant $C$ is independent of the indices $a$, $b$.
Combining \eqref{eq:I.inter2-MLSchroEqu2018}, \eqref{eq:hotF1F1J1-MLSchroEqu2018}, \eqref{eq:hotF1F1J2-MLSchroEqu2018} and \eqref{eq:hotF1F1J3-MLSchroEqu2018}, we can rewrite \eqref{eq:I.inter2-MLSchroEqu2018} as
\begin{align}
k_1 k_2 |\mathbb K(z,y)|
& \lesssim \iint_{\widetilde{\mathcal D} \times \widetilde{\mathcal D}} \big[ |s-y|^{-2} |t-z|^{-2} |s - t|^{-\tau} + |s-y|^{-2} |t-z|^{-1} |s - t|^{-1-\tau} \nonumber \\
& \hspace*{1cm} + |s-y|^{-1} |t-z|^{-2} |s - t|^{-1-\tau} + |s-y|^{-1} |t-z|^{-1} |s - t|^{-2-\tau} \big] \dif s \dif t \nonumber \\
& =: \mathbb I_1 + \mathbb I_2 + \mathbb I_3 + \mathbb I_4. \label{eq:bbI1234-MLSchroEqu2018}
\end{align}
{Denote $\mathbf D := \{x + x', x - x' \,;\, x, x' \in \widetilde{\mathcal D} \}$.
Then we apply Lemma \ref{lem:intbdd-MLSchroEqu2018} to estimate $\mathbb I_1$ as follows,}
\begin{align}
\mathbb I_1
& = \iint_{\widetilde{\mathcal D} \times \widetilde{\mathcal D}} |s-y|^{-2} |t-z|^{-2} |s - t|^{-\tau} \dif s \dif t \nonumber \\
& \leq \int_{\mathbf D} |s|^{-2} \big( \int_{\mathbf D} |t|^{-2} |t - (s+y-z)|^{-\tau} \dif t \big) \dif s \nonumber \\
& \lesssim C_{\widetilde{\mathcal D}} + \int_{\mathbf D} |s|^{-2} |s - (z-y)|^{-(\tau-1)} \dif s \nonumber \\
& \simeq |z-y|^{2-\tau} + C_{\widetilde{\mathcal D}}.
\label{eq:bbI1-MLSchroEqu2018}
\end{align}
Note that in \eqref{eq:bbI1-MLSchroEqu2018} we used Lemma \ref{lem:intbdd-MLSchroEqu2018} twice.
Similarly,
\begin{equation} \label{eq:bbI2-MLSchroEqu2018}
\mathbb I_2,\, \mathbb I_3,\, \mathbb I_4
\lesssim |z-y|^{2-\tau} + C_{\widetilde{\mathcal D}}.
\end{equation}
Recall that $\tau \in (0,1)$. By \eqref{eq:bbI1234-MLSchroEqu2018}, \eqref{eq:bbI1-MLSchroEqu2018} and  \eqref{eq:bbI2-MLSchroEqu2018} we arrive at
\begin{equation*}
|\mathbb K(z,y)|
\leq C k_1^{-1} k_2^{-1} (|z-y|^{2-\tau} + C_{\widetilde{\mathcal D}}) 
\leq C k^{-2} ((\diam D_V)^{2-\tau} + C_{\widetilde{\mathcal D}})
\lesssim k^{-2}.
\end{equation*}

\section{Conclusions} \label{sec:conclusion-RSRe2020}

We have reviewed the recoveries of some statistics by using near-field data as well as far-field data generated under a single realization of the randomness.
In this paper we mainly focus on time-harmonic Schr\"odinger systems.
One of the possible ways to extend the current works is to study the Helmholtz systems.
It would be also interesting to conduct the work in the time domain.
Moreover, the stability of the recovering procedure is also worth of investigation.

%
%


{

}

\end{document}